\documentclass[reqno, 12pt]{amsart}
\usepackage{graphicx} 
\usepackage[utf8]{inputenc}
\usepackage[british]{babel}
\usepackage[OT2, T1]{fontenc}
\usepackage{lmodern}
\usepackage{amsmath,amssymb,amsthm,todonotes,booktabs,graphicx,longtable, bbm, nicefrac, xfrac, adjustbox}
\usepackage{appendix}
\usepackage{tikz-cd}
\usepackage{url, hyperref}
\usepackage[margin=1.2in, hmarginratio=1:1]{geometry}
\definecolor{dblue}{rgb}{0,0,0.7}
\definecolor{dartmouthgreen}{rgb}{0.05, 0.5, 0.06}
\definecolor{dlilac}{rgb}{0.6, 0.33, 0.73}
\definecolor{ffuchsia}{rgb}{0.96, 0.0, 0.63}
\definecolor{hpurple}{rgb}{0.32, 0.09, 0.98}
\newtheorem{thm}{Theorem}
\newtheorem{prop}[thm]{Proposition}
\newtheorem{cor}[thm]{Corollary}
\newtheorem{lem}[thm]{Lemma}

\numberwithin{thm}{section}
\numberwithin{equation}{section}

\theoremstyle{definition}
\newtheorem{defi}[thm]{Definition}
\newtheorem{rem}[thm]{Remark}
\newtheorem{exa}[thm]{Example}

\newcommand{\QQ}{\mathbb{Q}}
\newcommand{\ZZ}{\mathbb{Z}}
\newcommand{\RR}{\mathbb{R}}
\newcommand{\CC}{\mathbb{C}}

\newcommand{\AAA}{\mathbb{A}}

\DeclareMathOperator{\im}{Im} 
\DeclareMathOperator{\re}{Re} 

\DeclareMathOperator{\Res}{Res}

\DeclareMathOperator{\Hom}{Hom}
\DeclareMathOperator{\aut}{Aut}

\DeclareMathOperator{\Gal}{Gal}

\DeclareMathOperator{\br}{Br} 

\DeclareMathOperator{\ab}{ab}
\DeclareMathOperator{\inv}{inv}
\DeclareMathOperator{\unr}{nr}

\DeclareMathOperator{\Frob}{Frob}
\DeclareMathOperator{\Ker}{Ker}

\DeclareMathOperator{\Sel}{Sel}

\DeclareSymbolFont{cyrletters}{OT2}{wncyr}{m}{n}
\DeclareMathSymbol{\Sha}{\mathalpha}{cyrletters}{"58}
\vspace{10cm}
\vspace{50cm}
\vfill
\newpage

\title{Counting abelian number fields with restricted ramification type}
\author{Julie Tavernier}
\date{July 2025}
\address{Julie Tavernier, Department of Mathematical Sciences, University of Bath,
Claverton Down, Bath, BA2 7AY, UK}
\email{jlt86@bath.ac.uk}

\begin{document}
  \begin{abstract} We count abelian number fields ordered by arbitrary height function whose generator of tame inertia is restricted to lie in a given subset of the Galois group, and find an explicit formula for the leading constant. We interpret our results as a version of the Batyrev-Manin conjecture on $BG$ and rephrase our result on number fields with restricted ramification type in terms of integral points on $BG$. We also prove that such number fields are equidistributed with respect to suitable collections of infinitely many local conditions. 
\end{abstract}
\maketitle

\tableofcontents

\section{Introduction}

A conjecture of Malle \cite{malle1, malle2} states that the number of number fields of bounded discriminant and fixed Galois group $G$ should satisfy an asymptotic formula of the form \[c_{k,G, \text{Malle}}B^{a(G)}(\log B)^{b(k,G) - 1}.\] There is an extensive literature on examples of this conjecture, for example the case of abelian number fields proved by Wright in \cite{Wright1989DistributionOD} and the cases of extensions of degree four and five whose Galois closures have Galois groups $S_4$ and $S_5$, proved by Bhargava in \cite{bhargavaquartic} and \cite{bhargavaquintic} respectively. However, the conjecture remains open in general. Recently, substantial work has taken place relating Malle's conjecture to the Batyrev-Manin conjecture for rational points on Fano varieties through the language of stacks \cite{loughransantens,ellenbergmaninheights, darda2024batyrevmaninconjecturedmstacks}. This viewpoint has given fundamental new insight into the conjecture by Malle.

The main contributions of this paper are as follows: first, we prove a generalisation of Malle's conjecture for abelian groups, allowing more general height functions, as proposed by Loughran and Santens in \cite[Conj.~$9.1$]{loughransantens}, broadening the work of Wood \cite{Wood_2010}. Secondly, we prove a refined version of this result counting such fields with restricted ramification type. Thirdly, we prove an equidistribution result for this count, in the vein of the Malle-Bhargava heuristics \cite{bhargavamassformulae, Wood_2010, bhargava2015geometryofnumbersmethodsglobalfields} and using this result we are able to prove a stacky analogue of a strong version of the Grunwald problem with restricted ramification type.

\subsection{Heights and restricted ramification type} The original conjecture posed by Malle concerns counting number fields ordered by discriminant, which is an example of a \emph{height} function. In \cite[Ques.~$4.3$]{ellenbergfunctfields}, Ellenberg and Venkatesh consider an extension of Malle's conjecture to counting by a different height function, which they call the $f$-discriminant. In this paper we count by arbitrary heights, using the definition of height functions found in \cite[$\S 8.1$]{loughransantens}.

In order to state our results, we will introduce some notation. Let $G$ be a finite abelian group and $k$ a number field, with a fixed algebraic closure $\overline{k}$ and absolute Galois group $\Gamma_k$. Recall that a surjective homomorphism $\varphi:\Gamma_k\to G$ corresponds to a finite Galois extension $L/k$, together with a fixed isomorphism $\Gal(L/k)\cong G$. We also recall the definition of the  Tate twist of $G$ by $-1$, which is $G(-1) = \Hom(\lim\limits_{\leftarrow n}\mu_n,G)$, and $G(-1)$ is viewed as a finite Galois module. We will denote the non-identity elements of $G(-1)$ by $G(-1)^{*}$.

Height functions are defined using the \emph{ramification type} which is the local map $\rho_{G,v} : \Hom(\Gamma_{k_v}, G) \rightarrow G(-1)^{\Gamma_{k_v}}$ defined in $\S$\ref{ht fns inerti sec} for tame places $v$. The ramification type $\rho_{G,v}(\varphi_v)$ of the homomorphism $\varphi_v \in\Hom(\Gamma_{k_v}, G) $ factoring through a finite, tamely ramified extension $L/k_v$ of ramification degree $e$, is the element of $G(-1)^{\Gamma_{k_v}}$ obtained by composing $\varphi_v$ with the canonical isomorphism $\mu_e \rightarrow I_v$, where $I_v$ is the inertia subgroup of $\Gal(L/k_v)$.

Let $w: G(-1) \rightarrow \QQ$ be a class function invariant under the action of $\Gamma_k$ and such that $w(1) = 0$, called a \emph{weight}  function. If $\varphi_v \in\Hom(\Gamma_{k_v}, G) $, the local height of $\varphi_v$ at all but finitely many tame places $v$ of $k$ is defined by \[H_v(\varphi_v) = q_v^{w(\rho_{G,v}(\varphi_v))},\] where $q_v$ is the cardinality of the residue field at $v$. For the finitely many remaining places we take $H_v$ to be arbitrary. The height $H$ is the product $H = \prod_v H_v$. We say $H$ is \emph{big} if its associated weight function satisfies $w(\gamma) > 0$ for all $\gamma \in G(-1)^*$. For each place $v$ of $k$, we fix an embedding of the algebraic closure of $k$ in that of $k_v$, and hence an inclusion $\Gamma_{k_v}\subseteq \Gamma_k$. Then given a homomorphism $\varphi \in \Hom (\Gamma_k, G)$, we denote by $\varphi_v$ its restriction to $\Gamma_{k_v}$. One of the first theorems of our paper is the following. 

\begin{thm}\label{main thm abt restr ram}
Let $H$ be a big height function with associated weight function $w$, $R \subseteq G(-1)^*$ be a non-empty, Galois-stable subset and $M_R(H) = \{\gamma \in R : w(\gamma) \: \text{is minimal}\}$. Let $S$ be a finite set of places of $k$ containing the infinite places, those dividing $|G|$ and such that $\mathcal{O}_S$ has trivial class group. Then we have that
        \begin{align*}
             \# \{\varphi \in \Hom (\Gamma_k, G): \:H(\varphi)\leq B, \: \rho_{G,v}(\varphi_v) \in & R \cup \{0\} \: \forall v \not\in S\} \\
              & \sim c_{k,R,G,H} B^{a_R(H)}(\log B)^{b_R(H)-1} \end{align*} where $a_R(H) = (\min_{\gamma \in R}w(\gamma))^{-1}$, the exponent of $\log B$ is $b_R(H) = | M_R(H)/ \Gamma_k|$ and $c_{k,R,G,H}$ is a positive constant. An explicit formula for $c_{k,R,G,H} $ in terms of sums of Euler products in the case that $M_R(H)$ generates the group $G$ is stated in (\ref{lc c(k,R,G,H) euler}), and a formula for $c_{k,R,G,H}$ in the general case is given by (\ref{lc unbalanced g ext restr ram}).
\end{thm} \noindent An explicit stacky formulation of the leading constant is stated in Theorem \ref{stacky main thm intro}.

We prove Theorem \ref{main thm abt restr ram} using techniques taken from harmonic analysis, developed in \cite[$\S 3$]{HNPabelianext} and \cite[$\S 3.4$]{Frei_2022}. We say that the homomorphisms $\varphi \in \Hom (\Gamma_k, G)$ in Theorem \ref{main thm abt restr ram} have ramification type restricted to lie in $R$, and note that this theorem implies that there are infinitely many such homomorphisms. Problems of this type arise in the Cohen-Lenstra heuristics. 

Number fields correspond to the surjective homomorphisms $\varphi \in \Hom (\Gamma_k, G)$, which we refer to as $G$-extensions, but we count all homomorphisms $ \Gamma_k \rightarrow G$. However, the non-surjective homomorphisms end up being negligible, which is seen by applying Theorem \ref{balanced ht restr ram thm} to each subgroup of $G$.

 When proving Theorem \ref{main thm abt restr ram}, it is important to consider separately the case of \emph{balanced} and \emph{unbalanced} height functions. The difference between these two cases was first observed by Wood in \cite{Wood_2010} where she considers the discrepancy that occurs when counting by discriminant instead of by conductor (note that she refers to what we call height functions as counting functions, and she uses the slightly different condition of fair counting functions). However one should note that the condition of a height being balanced in our sense is strictly weaker than the definition of fairness found in \cite[$\S 2.1$]{Wood_2010}.

In the special case where $R$ is empty, we recover those abelian number fields that are unramified outside of a finite set of places $S$. However, as there are only finitely many such fields,  counting them does not give any interesting results and so we consider the case where $R$ is non-empty. From Theorem \ref{main thm abt restr ram} we obtain a result on the total count of $G$-extensions of bounded height.

\begin{cor}\label{formal consequence}
    Let $H$ be a height function with associated weight function $w$. Let $M_G(H) = \{\gamma \in G(-1)^*: w(\gamma) \: \text{is minimal}\}$. Then  we have \[\#\{\varphi \in \Hom (\Gamma_k, G) : H(\varphi) \leq B\} \sim  c_{k,G,H} B^{a(H)} (\log B)^{b(H)-1}\] where \[a(H) = \left(\min_{\gamma \in G(-1)^*} w(\gamma)\right)^{-1} \:\ \text{and } \:\ b(H) = |M_G(H)/ \Gamma_k|\] and has an explicit formula in terms of sums of Euler products.
\end{cor} \noindent When $G$ is generated by $M_G(H)$, the leading constant $ c_{k,G,H}$ is given by (\ref{lc c(k,G,H) euler}). For general height functions, it is obtained from the formula (\ref{lc unbalanced g ext restr ram}) by setting $R = G(-1)^*$.

While special cases of this are known, such as for the conductor by Wood in \cite{Wood_2010} and the discriminant by Wright in \cite{Wright1989DistributionOD}, it is new for general height functions.  We begin by proving our theorems for balanced height functions and then consider the case where the height is not balanced. In order to count $G$-extensions in this case, we  apply the dominated convergence argument due to Rome and Koymans in \cite[Step~$1$]{koymansdct}.

\subsection{Counting number fields via stacks}

There is a natural interpretation of our problem of counting number fields via counting rational points on the stack $BG$ for a finite abelian group $G$, explored in the recent works \cite{ellenbergmaninheights, yasudabddtorsors, darda2024batyrevmaninconjecturedmstacks}, with a conjecture for the leading constant put forward in \cite{loughransantens}. 

The groupoid $BG(k)$ corresponds to $\Hom(\Gamma_k,G)$ and we denoted by $BG[k]$ the set of isomorphism classes of elements in $BG(k)$. From \cite[Def.~$7.2$]{loughransantens}, for all but finitely many tame places $v$ we have the \emph{partial $v$-adic integral points with respect to $R$} \[BG(\mathcal{O}_v)_{R} = \{\varphi_v \in BG(k_v) : \rho_{G,v}(\varphi_v) \in R \cup \{0\} \}\] and the \emph{partial adelic space with respect to $R$} \[BG(\AAA_k)_{R} = \lim_S \prod_{v \in S}BG(k_v)\prod_{v \not\in S}BG(\mathcal{O}_v)_{R},\] where the limit is taken over all finite sets of places $S$ containing the infinite places and those dividing $|G|$. 
By \cite[Thm.~$7.4$]{loughransantens}, an element $b \in \br BG$ is in the partially unramified Brauer group $ \br_{R}BG$ if and only if $b$ evaluates trivially on $BG(\mathcal{O}_v)_{R}$ for all but finitely many $v$. We write $BG(\AAA_k)_{M_R(H)}^{\br}$ for the Brauer-Manin set of $BG(\AAA_k)_{M_R(H)}$ with respect to $\br_{M_R(H)}BG$ and we use the definition from \cite[Def.~$8.1$]{loughransantens} that a \emph{balanced} height function is one such that the minimal elements $M_G(H)$ with respect to the weight function $w$ generate all of $G$. We call a height function \emph{balanced with respect to $R$} if $G$ is generated by $M_R(H)$. Consider the set $ \prod_{v \in S} BG(k_v) \prod_{v \not\in S}BG(\mathcal{O}_v)_{R}$. Asking that $\varphi \in  \prod_{v \in S} BG(k_v) \prod_{v \not\in S}BG(\mathcal{O}_v)_{R}$ is the same as asking that $\rho_{G,v}(\varphi_v) \in R \cup \{0\}$ for all but the finitely many places in $S$. Using the viewpoint of Manin's conjecture and the framework from \cite{loughransantens} we have the following stacky version of Theorem \ref{main thm abt restr ram}.
\begin{thm} \label{stacky main thm intro}
    Let $R \subseteq G(-1)^*$ be a non-empty Galois-stable subset, $S$ be a finite set of places containing the infinite places, those dividing $|G|$ and such that $\mathcal{O}_S$ has trivial class group, and $H$ be a big balanced height function with respect to $R$. Let \[W_{R,S} = \prod_{v \in S} BG(k_v) \prod_{v \not\in S}BG(\mathcal{O}_v)_{R}.\] Then we have 
\[
\frac{1}{|G|}\# \{ \varphi \in BG[k] :\varphi \in W_{R,S}, H(\varphi) \leq B\} \\ \sim c(k,R,G,H)B^{a_R(H)}(\log B)^{b_R(H)-1}\] where $a_R(H)$ and $b_R(H)$ are as in Theorem \ref{main thm abt restr ram} and $c(k,R,G,H)$ is given by \[ \frac{a_R(H)^{b_R(H)-1} |\br_{M_R(H)}BG/\br k|\tau_H\left( W_{R,S} \cap BG(\AAA_k)_{M_R(H)}^{\br}\right)}{|\widehat{G}(k)| (b_R(H)-1)!}.\] Here $\tau_H$ is a Tamagawa measure on $BG(\AAA_k)_{M_R(H)}$ defined in $\S$\ref{tamagawa measure sec} and $\widehat{G}$ is the Cartier dual of $G$. 
\end{thm} \noindent In the case $R=G(-1)^*$, this proves \cite[Conj.~$9.1$]{loughransantens} for finite abelian groups $G$.

\subsection{Equidistribution} \label{intro eq sec}

The Malle-Bhargava heuristics give a heuristic for the asymptotic behaviour of the quotient of the number of number fields of bounded height satisfying certain local specifications by the total count of number fields of bounded height. However, this is lacking as it does not take into account possible Brauer-Manin obstruction. A precise equidistribution conjecture is put forward in \cite[$\S 9.6$]{loughransantens} which takes into account such obstructions. We prove a strong version of equidistribution where we impose infinitely many local specifications. We consider the quotient of the count in Theorem \ref{main thm abt restr ram} in the case that the height function is balanced with respect to $R$, where we have imposed restrictions on the ramification type locally at all but finitely many places of $k$, by the total number of $G$-extensions of bounded height by interpreting it in the language of Manin's conjecture. Recall that a continuity set is one whose boundary has measure zero, in this case with respect to the Tamagawa measure.\begin{thm}[Strong equidistribution] \label{main thm}
   Let $G$ be a finite abelian group, $H$ be a big balanced height with respect to a non-empty Galois-stable subset $R \subseteq G(-1)^*$ and let $W \subset BG[\AAA_k]_{M_R(H)}$ be a continuity set. Then we have  \[\lim_{B \rightarrow \infty} \frac{\# \{ \varphi \in BG[k]: \varphi \in W, H(\varphi) \leq B\}}{\# \{ \varphi \in BG[k] : H(\varphi) \leq B\}} = \frac{\tau_{H}(W \cap BG[\AAA_k]_{M_R(H)}^{\br})}{\tau_{H}(BG[\AAA_k]_{M_R(H)}^{\br})}. \]
\end{thm} \noindent The boundary of a set has measure zero if the indicator function can be approximated above and below by continuous functions, as stated in the proof of \cite[Prop.~$5.0.1$(d)]{peyretamagawa}. In a more general setting, the indicator function of a continuity set can be approximated by a dense collection of continuous functions. In the case of $W \subset BG[\AAA_k]_{M_R(H)}$, $W$ is such that $\prod_v U_v \subseteq W \subseteq \prod_v V_v$ for multiplicatively defined sets, and finite linear combinations of the indicator functions of these sets are dense, and thus we may use them to approximate $W$.

When $R= G(-1)^*$, this proves \cite[Conj.~$9.15$]{loughransantens} in the case of a finite abelian group $G$. The proof of this theorem is obtained from  the stacky version of the leading constant in Theorem \ref{stacky main thm intro}. It is important to note that the assumption that $H$ is balanced with respect to $R$ is necessary for Theorem \ref{main thm} to hold. As observed by Wood in \cite[$\S 6$]{Wood_2010}, equidistribution need not hold when counting all extensions. Instead we restrict to the fibres of the map $BG \rightarrow BG/\langle M_R(H) \rangle $ to obtain equidistribution with respect to the induced Tamagawa measure on each fibre. 

\subsection{Application to the Grunwald problem}

We give an application of Theorem \ref{main thm} to a strong version of the Grunwald problem with restricted ramification type. Let $G$ be a finite abelian group, $k$ a number field and $S$ a finite set of places of $k$. The Grunwald problem asks whether the map \[\Hom(\Gamma_k,G) \rightarrow \prod_{v \in S} \Hom(\Gamma_{k_v},G)\] is surjective. In our setting of stacks, this can be rephrased as asking whether the map \[BG[k] \rightarrow \prod_{v \in S} BG[k_v]\] is surjective. This is not true in general due to possible Brauer-Manin obstruction. For example when $G = \ZZ/8\ZZ$ the map $B\ZZ/8\ZZ[\QQ] \rightarrow B\ZZ/8\ZZ[\QQ_2]$ is not surjective, which is the Grunwald-Wang Theorem. We prove a stronger version of the Grunwald problem where as well as approximating finitely many local conditions, we restrict the ramification type at almost all places. We also take into account the Brauer-Manin obstruction.

\begin{cor}\label{grunwald}
    Let $R \subseteq G(-1)$ be a Galois-stable subset containing a non-trivial element such that the elements of $R$ generate $G(-1)$, and $S$ be a finite set of places of $k$. Let $\br := \br_{R}BG$ be the partially unramified Brauer group with respect to $R$. Then the image of the map \[BG[k] \rightarrow BG[\AAA_k]_{R}^{\br}\] is dense. In particular there is an affirmative answer to the Grunwald problem with restricted ramification type for $G$, $R$ and $S$, and with Brauer-Manin obstruction.
\end{cor}

\noindent Corollary \ref{grunwald} follows immediately from Theorem \ref{main thm} by choosing a balanced height function such that $M_R(H) = R$. This is known in the case $R= G(-1)$ which is essentially Grunwald-Wang \cite[Thm.~$9.2.8$]{neukirch2013cohomology}, but is otherwise new and not yet studied in the literature.

  \begin{rem}
One can view Theorem \ref{main thm abt restr ram} as a version of Manin's conjecture for partial integral points on $BG$. This can be done by viewing the ``boundary divisors'' of $BG$ as the disjoint union of non-trivial conjugacy classes of $G(-1)$, which in the case of a constant abelian group $G$ is given by distinct non-trivial elements of $G(-1)$. For varieties there is an integral points version of Manin's conjecture due to Santens in \cite{santens2023maninsconjectureintegralpoints}. Our results can be put into this framework and are seen to agree with \cite[Conj.~$6.1$]{santens2023maninsconjectureintegralpoints}. In our case there is no corresponding Clemens complex. Then we interpret the problem of counting number fields with restricted ramification type as a version of Manin's conjecture for partial integral points on $BG$.
   \end{rem}  
  \subsection{Examples}
  We use the following example to make clear the difference between the conditions of balanced and fair heights.
\begin{exa}
    Let $G = \ZZ/6\ZZ$ and $k$ be a number field containing the sixth roots of unity. Let $T =\{1,5\} \subset G$. Then consider the height function $H$ whose corresponding weight function $w$ sends the elements of $T$ to $1$ and all remaining non-trivial elements of $G$ to $2$. This height is balanced as the minimal weight elements $\{1,5\}$ generate all of $G$ but $T \cap G[2]$ does not generate $G$, and thus is not fair in the sense of Wood in \cite[$\S 2.1$]{Wood_2010}. We may apply our Theorem \ref{balanced ht restr ram thm} to this and obtain an asymptotic formula and in particular an explicit expression for the leading constant. However, the leading constant in this case is given by a single Euler product rather than a sum of Euler products. This follows from \cite[Lem.~$6.32$]{loughransantens}, which gives that the relevant Brauer group is constant and hence there is no Brauer-Manin obstruction in this case. In particular we can then apply Theorem \ref{main thm} with local conditions and obtain equidistribution but with no Brauer-Manin obstruction.
\end{exa}
We now finish by making explicit our results in a special case to demonstrate the range of phenomenon which can appear and the role played by the Brauer group. We consider a height function proposed by Wood and presented by Alberts in \cite[$\S 7.6$]{alberts2023randomgrouplocaldata} as a height which exhibits pathological behaviour. This height is once again balanced in our case but not fair, and we explain its unexpected behaviour via a Brauer-Manin obstruction. An overview of this example is given here, with full details provided in $\S 4.5$.
  \begin{exa}
      \label{example in intro}
      Let $G = \ZZ/4\ZZ$ and $k = \QQ$. Let $H$ be a height function  which determines a local height at $p \neq 2,\infty$ with weight function \[w: 1 \mapsto 1, \: 2\mapsto 2, \: 3 \mapsto 1.\] At the places $v = \infty, 2 $ we take $H_v(\chi_v) = 1$. We explain the following special case. \\
          $(1)$ Let $R=\ZZ/4\ZZ(-1)$. Then Theorem \ref{main thm} gives the following asymptotic formula: \[ \frac{1}{4}\# \{\varphi \in B\ZZ/4\ZZ[\QQ] : H(\varphi) \leq B, \: \varphi \: \text{is completely split at $2$ and $\infty$}\} \sim cB\] where the leading constant is given by  \begin{align*}
               c = & \frac{1}{64}\prod_{\substack{p \: \text{prime}  \\ p \equiv 1 \bmod 4}}\left(1 - \frac{1}{p}\right)\left(1+\frac{2}{p}+ \frac{1}{p^2}\right) \prod_{\substack{p \: \text{prime} \\ p \equiv 3 \bmod 4}}\left(1 - \frac{1}{p}\right)\left(1+\frac{1}{p^2}\right) \\
               + & \frac{1}{64} \prod_{\substack{p \: \text{prime}  \\ p \equiv 1 \bmod 4}}\left(1 - \frac{1}{p}\right)\left(1+\frac{2}{p}+ \frac{1}{p^2}\right)\prod_{\substack{p \: \text{prime}  \\ p \equiv 3 \bmod 4}}\left(1 - \frac{1}{p}\right)\left(1-\frac{1}{p^2}\right).
          \end{align*} We impose the local conditions at $2$ and $\infty$ to simplify the local densities at these places. Let $ \br_eB\ZZ/4\ZZ = \{b \in \br B\ZZ/4\ZZ : b(e) = 0\}$ where $e$ is the identity. We show in Lemma \ref{unr br grp ex} that the relevant Brauer group $\br_{e,\{1,3\}}B(\ZZ/4\ZZ) = \br_{\{1,3\}}B(\ZZ/4\ZZ) \cap \br_e B\ZZ/4\ZZ$ has two elements, where the non-trivial element corresponds to $-4 \in \br_eB\ZZ/4\ZZ = H^{1}(\QQ, \mu_4)$. This Brauer group element gives a Brauer-Manin obstruction, which can be described as follows: let $S$ be a finite set of primes of $\QQ$ not containing $2$ or $\infty$ such that the number of primes $p \equiv 3 \bmod 4$ in $S$ is odd. Then there is no $\ZZ/4\ZZ$-extension of $\QQ$ completely split at $2$ and $\infty$ such that $\rho_{\ZZ/4\ZZ,p}(\varphi) \in \{1,3\}$ for all $p \not\in S$ and $\rho_{\ZZ/4\ZZ,p}(\varphi) = 2$ for $p \in S$. This obstruction explains why the leading constant is given as the sum of two convergent Euler products. In particular, the $-\frac{1}{p^2}$ term in the second Euler product arises as a result of this Brauer-Manin obstruction. \\
          $(2)$ Now let $R= \{1,3\}$ in Theorem \ref{main thm}. In this case we count those $\ZZ/4\ZZ$-extensions of $\QQ$ whose ramification type is trivial or lies in $R$ and that are completely split at $2$ and $\infty$. In particular Theorem \ref{main thm} gives
          \[ \frac{1}{4} \left\{ \varphi \in B\ZZ/4\ZZ[\QQ] : \begin{array}{ll}
               & H(\varphi) \leq B, \: \rho_{\ZZ/4\ZZ,p}(\varphi_p) \in \{0,1,3\} \: \text{for} \: p \neq 2, \infty  \\
               & \varphi \: \text{is completely split at $2$ and $\infty$}
          \end{array} \right\}  \sim c_{ \{1,3\}} B.\] In this case the minimal weight elements with respect to $R$ are given by $M_R(H) = \{1,3\}$ and generate $\ZZ/4\ZZ$, thus we have that $H$ is balanced with respect to $R$ and Theorem \ref{main thm} can be applied in this case. We then obtain the formula for the leading constant \[ c_{ \{1,3\}} = \frac{1}{32} \prod_{\substack{p \: \text{prime}  \\ p \equiv 1 \bmod 4}}\left(1-\frac{1}{p}\right)\left(1+\frac{2}{p}\right)\prod_{\substack{p \: \text{prime}  \\ p \equiv 3 \bmod 4}}\left(1-\frac{1}{p}\right).\] The relevant Brauer group in this case is the same as that in Part $(1)$, however there is no Brauer-Manin obstruction coming from the element $-4$ when imposing restrictions on the ramification type via the set $R$. In particular, the leading constant is given by a single Euler product. \\
          $(3)$ If instead $R= \{2\}$ in Theorem \ref{main thm abt restr ram}, then we have the asymptotic formula \[\frac{1}{4}\#  \{\varphi \in B\ZZ/4\ZZ[\QQ] : H(\varphi) \leq B, \:  \rho_{\ZZ/4\ZZ,p}(\varphi_p) \in \{0,2\} \: \text{for} \: p \neq \infty, 2\} \sim c_{\{2\}} B.\]Here no real or $2$-adic splitting conditions are imposed. The minimal weight elements with respect to $R$ are now $M_R(H)= \{2\}$, which does not generate $\ZZ/4\ZZ$. Thus $H$ is not balanced with respect to $R = \{2\}$ and we may not apply Theorem \ref{stacky main thm intro} directly. Instead we consider the fibres of the map $B\ZZ/4\ZZ \rightarrow B\ZZ/2\ZZ$, as by \cite[Lem.~$3.31$]{loughransantens} $H$ is balanced when restricted to each fibre. This map associates to each $\ZZ/4\ZZ$-extension of $\QQ$ with restricted ramification type imposed by $R$ its unique quadratic subfield. We then sum over the fibres. In our case, only one fibre contributes and this fibre corresponds to counting quadratic extensions of $\QQ(\sqrt{2})$. The leading constant turns out to be 
       \[c_{\{2\}} =  \frac{\lim_{s \rightarrow 1}(s-1)\zeta_{\QQ(\sqrt{2})}(s)}{128} \prod_v\left(1-\frac{1}{q_v}\right)\left(1 + \frac{1}{q_v}\right),\] \noindent where the product is taken over all places of $\QQ(\sqrt{2})$. 
 \end{exa}

\subsection{Structure}

In $\S$\ref{counting no fields restr ram} we use techniques taken from harmonic analysis to prove Theorem \ref{main thm abt restr ram} in the case where the height function is balanced, along with an explicit formula for the leading constant in terms of sums of Euler products. In $\S 3$ we extend our results to general height functions using the dominated convergence argument from Step $1$ of the proof of \cite[Thm.~$1.1$]{koymansdct}. In $\S$\ref{inter resulst via stacks} we define the necessary terminology needed to interpret our results via stacks, and introduce and describe the properties of the Tamagawa measure appearing in the leading constant. We include the definition of the Brauer-Manin pairing on $BG$ needed to define the partially unramified Brauer group in the leading constant. In this section we also prove a formula for the leading constant in terms of Tamagawa measures, using the framework from \cite{loughransantens}, and provide a proof for Example \ref{example in intro}. Finally in $\S$\ref{eq section} we prove our results on equidistribution.

\subsection{Notation}

Throughout this paper $k$ is a fixed number field. We use the following notation.

\begin{itemize}
    \item $\AAA^{\times}$ the ideles of $k$ 
    \item $k_v$ the completion of $k$ at a place $v$
    \item $\mathcal{O}_v$ the ring of integers of $k$ and for $v \mid \infty$ we use the convention $\mathcal{O}_v = k_v$
    \item $\Omega_k$ the set of places of $k$
    \item $q_v$ the size of the residue field at $v$
    \item $\zeta_k(s)$ the Dedekind zeta function of $k$
    \item $\widehat{G}$ the Cartier dual of $G$, $\widehat{G} = \Hom(G,\mathbbm{G}_m)$
    \item $G^{\wedge}$ the Pontryagin dual of $G$, $G^{\wedge} = \Hom(G, S^1)$.
\end{itemize} 

\subsection{Acknowledgements}

The author would like to thank Daniel Loughran for suggesting the project and for many helpful discussions, as well as Peter Koymans, Nicholas Rome, Julian Demeio, and Tim Santens for kindly answering my questions, and Ross Paterson for his help with the proof of Proposition \ref{existence of a lift}. I would also like to thank Brandon Alberts and Ross Paterson for useful comments.

\section{Counting number fields with restricted ramification type} \label{counting no fields restr ram}

\subsection{Ramification type and height functions }\label{ht fns inerti sec}
We define height functions, using the ramification type $\rho_{G,v} : \Hom(\Gamma_{k_v}, G) \rightarrow G(-1)^{\Gamma_{k_v}} $, as defined in \cite[$\S 7.1$]{loughransantens}. Essentially the same definition of the ramification type can be found in the work of Gundlach \cite[$\S 2$]{gundlach2022mallesconjecturemultipleinvariants}. Consider a homomorphism $\varphi_v \in \Hom(\Gamma_{k_v}, G)$, corresponding to a sub-$G$-extension of $k_v$. As $\varphi_v$ is continuous, for all but finitely many tame places it factors through a finite tamely ramified extension $L/k_v$ with ramification degree $e$. We may assume the extension $L$ contains $\mu_e$ (as we can enlarge it if necessary), and let $\varpi$ be a uniformiser of $L/k_v$. Denoting by $I_v$ the inertia group of $\Gal(L/k_v)$, by  \cite[Tag~$09$EE]{stacks-project} there is a canonical isomorphism  \[\mu_e \rightarrow I_v, \:\ \zeta \mapsto (\sigma_{\zeta}: \varpi \rightarrow \zeta \varpi)\] and we may compose this with $\varphi_v$ to obtain a homomorphism \begin{equation}\label{composition}
 \mu_e \rightarrow I_v \rightarrow G.
 \end{equation} In particular this gives a homomorphism $\mu_e \rightarrow G$, which is a Galois-invariant element of $G(-1)$, and we denote this by $\rho_{G,v}(\varphi_v)$.

 \begin{lem}\label{properties}
       Let $G$ be a finite abelian group and $\rho_{G,v} : \Hom(\Gamma_{k_v}, G) \rightarrow G(-1)^{\Gamma_{k_v}} $ be the ramification type. Then we have the following:
\begin{enumerate}
        \item $\rho_{G,v}$ is a homomorphism.
        \item The kernel of $\rho_{G,v}$ is exactly the unramified homomorphisms.
        \item $\rho_{G,v}$ induces an isomorphism \[\frac{\Hom(\Gamma_{k_v}, G)}{\Hom(\Gamma_{k_v}^{\unr}, G)} \xrightarrow{\sim} G(-1)^{\Gamma_{k_v}},\] where $\Gamma_{k_v}^{\text{nr}}$ is the Galois group of the maximal unramified extension of $k_v$.
    \end{enumerate}
 \end{lem}

 \begin{proof}Since $G$ is a finite abelian group, the Tate twist $G(-1)$ has a natural structure as a Galois module. Let $\varphi_1, \varphi_2 \in \Hom(\Gamma_{k_v},G)$. By continuity $\varphi_1$ factors through a finite extension $L_1/k_v$ of ramification degree $e_1$ and $\varphi_2$ factors through a finite extension $L_2/k_v$ of ramification degree $e_2$. As $G$ is a finite abelian group we have that $\varphi_1 + \varphi_2 \in \Hom(\Gamma_{k_v},G)$, and  $\varphi_1 + \varphi_2$ factors through the group $\Gal(L_1L_2/k_v)$ where $L_1L_2$ is tamely ramified and has ramification degree $e = \text{lcm}(e_1,e_2)$. Then the ramification type $\rho_{G,v}(\varphi_1 + \varphi_2 )$ is the induced map on the inertia subgroup of $\Gal(L_1L_2/k_v)$, which is canonically isomorphic to the group of roots of unity $\mu_e$. Since $\rho_{G,v}(\varphi_1)$ and $\rho_{G,v}(\varphi_2)$ agree on $\mu_{e_1}\cap \mu_{e_2}$ we have that there is a unique extension of the homomorphism $\rho_{G,v}(\varphi_1)\rho_{G,v}(\varphi_2 ): \mu_{e_1}\cap \mu_{e_2} \rightarrow G$ to the domain $\mu_{\text{lcm}(e_1,e_2)}$ such that $\rho_{G,v}(\varphi_1)\rho_{G,v}(\varphi_2)|_{\mu_{e_1}} = \rho_{G,v}(\varphi_1)$ and $\rho_{G,v}(\varphi_1)\rho_{G,v}(\varphi_2 )|_{\mu_{e_2}} = \rho_{G,v}(\varphi_2)$. But since this restriction condition is also satisfied for $\rho_{G,v}(\varphi_1 + \varphi_2 )$ and there is exactly one homomorphism for which it holds we have $\rho_{G,v}(\varphi_1 + \varphi_2 ) = \rho_{G,v}(\varphi_1 )\rho_{G,v}(\varphi_2)$. For Part $2)$ it follows immediately from the definition of $\rho_{G,v}$ that it is trivial on the unramified homomorphisms. On the other hand take $\varphi \in \Ker(\rho_{G,v})$. Then $\rho_{G,v}(\varphi)$ is trivial, and so is the induced map of $\varphi$ on the inertia group and hence it is unramified. In particular, $\varphi$ is unramified. For Part $3)$, since $\rho_{G,v}$ is trivial on $\Hom(\Gamma_{k_v}^{\text{nr}}, G)$ and there is an isomorphism $G(-1)^{\Gamma_{k_v}} \cong \Hom(I_{k_v}, G)$, it is enough to show that the map $\Hom(\Gamma_{k_v}, G) \rightarrow \Hom(I_{k_v}, G)$ is surjective. We consider the inertia exact sequence 
     \begin{center}
         
\begin{tikzcd}
	0 \arrow[r] & {I_{k_v}} \arrow[r, "i"] & {\Gamma_{k_v}} \arrow[r] & {\Gamma_{k_v}^{\text{nr}}} \arrow[r] & 0,
\end{tikzcd}   \end{center} 

\noindent where $i$ is the injection. Since we are only considering the abelian case, this sequence fits into the following commutative diagram of exact sequences of topological groups coming from local class field theory
    \begin{center}
\begin{tikzcd}
0 \arrow[r] 
&  {\mathcal{O}_v^{\times} } \arrow[r] \arrow[d] 
&  {k_v^{\times}}  \arrow[r] \arrow[d] 
& {\ZZ}  \arrow[r] \arrow[d] 
& 0  \\
0 \arrow[r]
& {I_{k_v}^{\ab}}\arrow{r}  &{\Gamma_{k_v}^{\ab}}\arrow[r]&{\Gamma_{k_v}^{\text{nr},\ab}} \arrow[r]   & 0.
\end{tikzcd}
     \end{center}
The top sequence splits and hence so does the one on the bottom. Applying $\Hom(\cdot, G)$ to the inertia exact sequence, we obtain the exact sequence \\
\begin{center}
\begin{tikzcd}
	0 & {\Hom(I_{k_v}, G)} \arrow[l] \arrow[bend right]{r}[black,swap]{r_{*}} & {\Hom(\Gamma_{k_v}, G)} \arrow[l, "i_{*}"] & {\Hom(\Gamma_{k_v}^{\text{nr}}, G)} \arrow[l]  & 0 \arrow[l],
\end{tikzcd} 
\end{center} where the leftmost arrow follows from the fact that map $\Hom(I_{k_v}, G) \rightarrow \text{Ext}(\Gamma_{k_v}^{\unr}, G)$ factors through zero. The map $r_{*}$ is a retraction and by a standard argument we obtain that $i_{*}$ is a surjection. \end{proof}

 \begin{defi}
     A \emph{weight function} is a class function $w:G(-1) \rightarrow \QQ$ satisfying $w(1) = 0$ and which is invariant under the action of $\Gamma_k$. A weight function $w$ is \emph{big} if $w(\gamma) > 0$ for all $\gamma \in G(-1)^*$.
 \end{defi}

\begin{defi}\label{ht fns}
An \emph{adelic height} $H = (H_v)_{v \in \Omega_k}$ is a collection of local maps $H_v: \Hom(\Gamma_{k_v},G) \rightarrow \RR$ such that for all but finitely many tame places $v$ and all $\varphi_v \in \Hom(\Gamma_{k_v},G)$ we have \[H_v(\varphi_v) =q_v^{w(\rho_{G,v}(\varphi_v))},\] where $\rho_{G,v}$ is the ramification type, and $w$ is a weight function. Then the height of $\varphi \in \Hom(\Gamma_k,G)$ is defined to be the product $H(\varphi) = \prod_v H_v(\varphi_v)$ of local heights over all places $v$. 
\end{defi} \noindent Recall that the set $M_G(H)$ is the set \[M_G(H) = \{1 \neq \gamma \in G(-1) : \: w(\gamma) \; \text{is minimal}\},\] and we say $H$ is \emph{balanced} if $G$ is generated by $M_G(H)$, and \emph{unbalanced} otherwise. A height function $H$ is \emph{big} if its associated weight function is big.
\subsection{Counting with restricted ramification type} \label{counting res sam sect}
From hereon out $R \subseteq G(-1)^*$ is a non-empty Galois-stable subset, $H$ is a big height function and $M_R(H) = \{\gamma \in R : \: w(\gamma) \: \text{is minimal}\}$. We fix $S$ to be a finite set of places containing the infinite places and those dividing $|G|$ and such that $\mathcal{O}_S$ has trivial class group. We call a height function $H$ \emph{balanced with respect to $R$} if $M_R(H)$ generates $G$. Let $\varphi_v \in \Hom(\Gamma_{k_v},G)$ be a sub-$G$-extension of $k_v$. For all $v \not\in S$, we will let $f_{R,v}$ be the indicator function for the condition that $\rho_{G,v}(\varphi_v) \in R \cup \{0\}$, and for the remaining finitely many places $v \in S$ we set $f_{R,v} = 1$, which is to say we will impose no conditions at these places. A $G$-extension of $k$ has ramification type restricted by $R$ if $\rho_{G,v}(\varphi_v) \in R \cup \{0\}$ for all $v \not\in S$ and a number field has restricted ramification type if this holds for corresponding $G$-extension. Then we define the adelic indicator function $f_R = \prod_v f_{R,v}$ and let $N(k,R,H,B)$ be the counting function \[N(k,R,H,B) = \# \{\varphi \in \Hom (\Gamma_k, G): H(\varphi) \leq B, \: \rho_{G,v}(\varphi_v) \in R \cup \{0\}\; \forall v \not\in S\}.\]This corresponds to the sum \begin{equation}\label{cnt func N(k,R,H,B)}
    N(k,R,H,B) = \sum_{\substack{\varphi \in \Hom (\Gamma_k, G) \\ H(\varphi) \leq B}} f_{R}(\varphi).
\end{equation} 
\subsection{The Fourier transforms} 

The function $N(k,R,H,B)$ has an associated height zeta function  \begin{equation}
    F_{R}(s) = \sum_{\varphi \in  \Hom(\Gamma_k,G)} \frac{f_{R}(\varphi)}{H(\varphi)^s}.
\end{equation}
\noindent Via the global Artin map we have the identification \[\Hom(\Gamma_k, G) = \Hom(\AAA^{\times}/k^{\times}, G)\] and via the local Artin map we have the identification \[\Hom(\Gamma_{k_v}, G) = \Hom(k_v^{\times}, G).\] The groups $\Hom(\AAA^{\times}/k^{\times}, G)$ and $\Hom(k_v^{\times}, G)$ are locally compact abelian groups and by \cite[Lem.~$3.2$]{HNPabelianext} their Pontryagin duals are identified with $ \AAA^{\times}/k^{\times} \otimes G^{\wedge}$ and $k_v^{\times} \otimes G^{\wedge}$ respectively. Moreover, they have associated pairings $\langle \cdot,\cdot \rangle: \Hom(\mathbb{A}^{\times}/k^{\times},G) \times (\AAA/k^{\times} \otimes G^{\wedge}) \rightarrow S^{1}$ and $\langle \cdot,\cdot \rangle:\Hom(k_v^{\times},G) \times( k_v^{\times} \otimes G^{\wedge} ) \rightarrow S^{1}$.

In an abuse of language we shall refer to the elements of $\Hom(k_v^{\times}, G)$ as characters. If $\chi_v \in \Hom(k_v^{\times}, G)$, then we write $H_v(\chi_v)$ to mean the height of the corresponding sub-$G$-extension. The \emph{inertia group of $\chi_v$} is the image of $\mathcal{O}_v^{\times}$ under the local Artin map. We say $\chi_v$ is \emph{unramified} if it is trivial on $\mathcal{O}_v^{\times}$. Clearly for $v \not\in S$, the function $f_{R,v}$ takes value $1$ on unramified elements and so does $H_v$. In particular, the function $f_{R}/H^s$ is the product of the local height functions $f_{R,v}/H_v^s$ which take value $1$ on the unramified elements. This means that the function can be extended to a well-defined continuous function on $\Hom(\AAA^{\times},G)$. We will also equip the group $\Hom(k_v^{\times},G)$ with the unique Haar measure d$\chi_v$ satisfying  \[\text{vol}(\Hom(k_v^{\times}/\mathcal{O}_v^{\times}, G)) = 1. \] For non-archimedean places $v$ we choose our measure to be such that  it is $|G|^{-1}$ times the counting measure and for archimedean $v$, we choose it to be the counting measure (using the convention that for archimedean places $\mathcal{O}_v = k_v$). Then the product of these local measures gives a well-defined measure d$\chi$ on $\Hom(\AAA^{\times},G)$. Thus we define the global Fourier transforms to be \[\hat{f}_{H,R}(x;s) = \int_{\chi \in \Hom(\AAA^{\times},G)} \frac{f_{R}(\chi)\langle \chi, x \rangle}{H(\chi)^s} d\chi\] for $x \in \AAA^{\times} \otimes G^{\wedge}$. Similarly we have the local Fourier transforms \[\hat{f}_{H,R,v}(x_v;s) = \int_{\chi_v \in \Hom(k_v^{\times},G)} \frac{f_{R,v}(\chi_v)\langle \chi_v, x_v \rangle}{H_v(\chi_v)^s} d\chi_v\] for $x_v \in k_v^{\times} \otimes G^{\wedge}$.\\

\noindent The global Fourier transform  exists for $\re(s) \gg 1$ and has an Euler product decomposition \begin{equation}\label{euler prod}
    \hat{f}_{H,R}(x;s) = \prod_v\hat{f}_{H,R,v}(x_v;s).\end{equation} By our choice of measure, for non-archimedean places we can write \[\hat{f}_{H,R,v}(x_v;s) = \frac{1}{|G|}\sum_{\chi_v \in \Hom(k_v^{\times}, G)} \frac{f_{R,v}(\chi_v)\langle \chi_v, x_v \rangle}{H_v(\chi_v)^s}.\] We need to ensure the local Fourier transforms satisfy the following analytic properties.
    \begin{lem}
        For $\re(s) \geq 0$ we have $\hat{f}_{H,R,v}(x_v;s) \ll_{k,H,R,G} 1$ and $\hat{f}_{H,R,v}(1;s) > 0$ for $s \in \RR_{>0}$.
    \end{lem}

    \begin{proof}
       We prove this in the non-archimedean case, as the archimedean case is analogous. We have \[\hat{f}_{H,R,v}(x_v;s) = \frac{1}{|G|}\sum_{\chi_v \in \Hom(k_v^{\times}, G)} \frac{f_{R,v}(\chi_v)\langle \chi_v, x_v \rangle}{H_v(\chi_v)^s}.\] Each summand satisfies $\leq q_v^{\frac{-s}{a_R(H)}}$, and the number of summands is $\ll_{k,G} 1$. For the second part, let $s\in \RR$. We have \[\hat{f}_{H,R,v}(1;s) =  \frac{1}{|G|}\sum_{\chi_v \in \Hom(k_v^{\times}, G)}\frac{f_{R,v}(\chi_v)}{H_v(\chi_v)^{s}}.\] If $v \in S$ then $f_{R,v}(\chi_v) = 1$ for all $\chi_v$ and this is clearly non-empty, positive and finite. On the other hand if $v \not\in S$ then the sum becomes \[\hat{f}_{H,R,v}(1;s) =  \frac{1}{|G|}\sum_{\substack{\chi_v \in \Hom(k_v^{\times}, G) \\ \rho_{G,v}(\chi_v) \in R \cup \{0\}}}\frac{1}{H_v(\chi_v)^{s}}\] and this is positive and finite. Furthermore the sum is non-empty as $\Hom(k_v^{\times}, G)$ always contains the trivial homomorphism which satisfies $\rho_{G,v}(\chi_v) = 1$.
    \end{proof}
\noindent Let $a_R(H) = (\min_{\gamma \in R}w(\gamma))^{-1}$. We calculate the local Fourier transforms appearing in the Euler product decomposition of $\hat{f}_{H,R}(x;s)$ for places $v \not\in S$.

\begin{lem}\label{restricted ram lft}
    Let $v \not\in S$ and let $x \in \mathcal{O}_v^{\times} \otimes G^{\wedge}$. Then there exists $\epsilon > 0$ such that
    \[\hat{f}_{H,R,v}(x_v;s) =  1 +  \sum_{\substack{\chi_v \in \Hom(\mathcal{O}_v^{\times},G)\\ \rho_{G,v}(\chi_v) \in M_R(H) \cup \{0\} \\ \chi_v \neq 1_v}}\langle \chi_v, x_v \rangle q_v^{-\frac{s}{a_R(H)}} +O_{\epsilon}\left(q_v^{-\left(\frac{1}{a_R(H)} + \epsilon \right)s}\right). \]
    In particular, if $x_v \in  \mathcal{O}_v^{\times} \otimes G^{\wedge}$ satisfies that $\langle \chi_v, x_v \rangle = 1$ for all $\chi_v$ such that $\rho_{G,v}(\chi_v) \in M_R(H) \cup \{0\}$ then we have \[\hat{f}_{H,R,v}(x_v;s) = 1 +|M_R(H)^{\Gamma_{k_v}}|  q_v^{-\frac{s}{a_R(H)}} + O_{\epsilon}\left(q_v^{-\left(\frac{1}{a_R(H)} + \epsilon \right)s}\right).\]
\end{lem}

\begin{proof}
     For $v \not\in S$, the functions $f_{R,v}$ and $H_v$ are $\Hom(k_v^{\times}/\mathcal{O}_v^{\times}, G)$-invariant, hence we may write the Fourier transform for $f_{R,v}/H_v^s$ as \[\hat{f}_{H,R,v}(x_v;s) = \sum_{\chi_v \in \Hom(\mathcal{O}_v^{\times},G)}\frac{f_{R,v}(\chi_v)\langle \chi_v, x_v \rangle }{H_v(\chi_v)^s} = 1 + \sum_{\substack{\chi_v \in \Hom(\mathcal{O}_v^{\times},G)\\ \rho_{G,v}(\chi_v) \in R \cup \{0\} \\ \chi_v \neq 1_v}}\frac{\langle \chi_v, x_v \rangle}{H_v(\chi_v)^s}.\] We split the sum into those $\chi_v$ such that $\rho_{G,v}(\chi_v) \in M_R(H)$ and those such that $\rho_{G,v}(\chi_v) \in R \backslash M_R(H)$. The $\chi_v$ whose ramification type $\rho_{G,v}(\chi_v)$ is non-minimal contribute $O_{\epsilon}\left(q_v^{-\left(\frac{1}{a_R(H)} + \epsilon \right)s}\right)$ to the sum for some $\epsilon > 0$. Then we obtain that the local Fourier transforms are equal to \[\hat{f}_{H,R,v}(x_v;s) = 1 +  \sum_{\substack{\chi_v \in \Hom(\mathcal{O}_v^{\times},G)\\ \rho_{G,v}(\chi_v) \in M_R(H) \cup \{0\} \\ \chi_v \neq 1_v}}\langle \chi_v, x_v \rangle q_v^{-\frac{s}{a_R(H)}} +O_{\epsilon}\left(q_v^{-\left(\frac{1}{a_R(H)} + \epsilon \right)s}\right) .\] Now suppose that $x_v$ is such that $\langle \chi_v, x_v \rangle = 1$ for all $\chi_v$ with $\rho_{G,v}(\chi_v) \in M_R(H) \cup \{0\}$. There is exactly one $\chi_v \in \Hom(\mathcal{O}_v^{\times},G)$ such that $\rho_{G,v}(\chi_v) = 1$, which is the trivial character (corresponding to the unramified homomorphism). Furthermore, by definition $\rho_{G,v}$ only takes values in the $\Gamma_{k_v}$-invariant elements of $G(-1)$. Thus \[ \sum_{\substack{\chi_v \in \Hom(\mathcal{O}_v^{\times},G)\\ \rho_{G,v}(\chi_v) \in M_R(H) \cup \{0\} \\ \chi_v \neq 1_v}}\frac{\langle \chi_v, x_v \rangle}{H_v(\chi_v)^s} = \sum_{\substack{\chi_v \in \Hom(\mathcal{O}_v^{\times},G)\\ \rho_{G,v}(\chi_v) \in M_R(H) \cup \{0\} \\ \chi_v \neq 1_v}} q_v^{-\frac{s}{a_R(H)}} =  \sum_{\gamma \in M_R(H)^{\Gamma_{k_v}}} q_v^{-\frac{s}{a_R(H)}} \sum_{\substack{\chi_v \in \Hom(\mathcal{O}_v^{\times}, G) \\ \rho_{G,v}(\chi_v) = \gamma}}1.\] It remains to show that for all $\gamma \in G(-1)$ there is exactly one $\chi_v \in \Hom(\mathcal{O}_v^{\times}, G)$ such that $\rho_{G,v}(\chi_v) = \gamma$. The fact that there is at most one follows from the fact that $\rho_{G,v}$ is well-defined, and the fact that there is exactly one follows from a diagram chase of the following commutative diagram to show that the map is surjective,

     \begin{center}
\begin{tikzcd}
&&& G(-1)^{\Gamma_{k_v}} \\
0 \arrow[r] 
&  {\Hom(\Gamma_{k_v}^{\text{nr}}, G)} \arrow[r] \arrow[d, "\sim" {rotate=90, anchor=north}] 
&  {\Hom(\Gamma_{k_v}, G)} \arrow[ur, "\rho_{G,v}"]  \arrow[r, two heads, "\beta"] \arrow[d, "\sim" {rotate=90, anchor=north}] 
& {\Hom(I_{k_v}, G)}  \arrow[r] \arrow[d, "\sim" {rotate=90, anchor=north}]  \arrow[u, "\sim" {rotate=90, anchor=north}]
& 0  \\
0 \arrow[r]
& {\Hom(\ZZ,G)} \arrow{r}  &{\Hom(k_v^{\times},G)}\arrow[r, two heads, "\tilde{\beta}"]&{\Hom(\mathcal{O}_v^{\times},G)} \arrow[r]   & 0.
\end{tikzcd}
     \end{center} where the isomorphism $\Hom(I_{k_v}, G) \cong G(-1)^{\Gamma_{k_v}}$ follows from Lemma \ref{properties}.
 \end{proof}

\subsection{The exponent} \label{b_R(H)}

We will use the theory of $S$-frobenian functions from \cite[$\S 3.3-3.4$]{serre2016lectures} to prove the expression for the exponent $b_R(H)$ in Theorem \ref{main thm abt restr ram}.\begin{defi}
    Let $k$ be a number field and $f : \Omega_k \rightarrow \CC$ a function on the set of places of $k$. Let $S$ be a finite set of places of $k$. We say that $f$ is an \emph{$S$-frobenian function} if there exists a finite extension $K/k$ with Galois group $\Gamma = \Gal(K/k)$ such that $S$ contains all places that ramify in $K/k$ and a class function $\varphi : \Gamma \rightarrow \CC$ satisfying \[\varphi(\Frob_{v}) = f(v)\] for all $v \not\in S$.
\end{defi} \noindent We define the mean of an $S$-frobenian function $f$ with corresponding class function $\varphi: \Gal(K/k) \rightarrow \CC$ to be \[m(f) = \frac{1}{[K:k]}\sum_{\sigma \in \Gal(K/k)} \varphi(\sigma).\] For $x \in \mathcal{O}_S^{\times} \otimes G^{\wedge}$ let $\lambda_x$ be the function

\begin{equation}\label{frob fn 2}
    \lambda_x(v) = \sum_{\substack{\chi_v \in \Hom(\mathcal{O}_v^{\times},G)\\ \rho_{G,v}(\chi_v) \in M_R(H) \cup \{0\} \\ \chi_v \neq \mathbbm{1}_v}}\langle \chi_v, x_v \rangle.
\end{equation} It follows from \cite[Rem.~$8.13$]{loughransantens} that $\lambda_1(v)$ is $S$-frobenian for our set $S$. It will then follow from \cite[Thm.~$8.23$]{loughransantens} that $\lambda_x(v)$ is also $S$-frobenian, and we denote its mean by $b_R(H,x)$. If $x$ satisfies that $\langle \chi_v, x_v \rangle = 1$ for all $\chi_v$ such that $\rho_{G,v}(\chi_v) \in M_R(H) \cup \{0\}$, this function becomes  \[\lambda(v) = |M_R(H)^{\Gamma_{k_v}}|.\]
\noindent We define $b_R(H)$ to be the mean of the function $v \mapsto |M_R(H)^{\Gamma_{k_v}}|$. Then $b_R(H,x) \leq b_R(H)$, with equality for the $x$ such that $\lambda_x = \lambda$. The class function corresponding to $\lambda$ is the function $\varphi : \Gal(k(\zeta_{|G|})/k) \rightarrow G$, given by \[ \varphi(\sigma) = |M_R(H)^{\sigma}|.\] We then calculate $b_R(H)$ as follows:

\begin{align*}
    b_R(H) = & \frac{1}{[k(\zeta_{|G|}): k]} \sum_{\sigma \in  \Gal(k(\zeta_{|G|})/k) } \varphi(\sigma) \\
    = & \frac{1}{[k(\zeta_{|G|}): k]} \sum_{\sigma \in  \Gal(k(\zeta_{|G|})/k) } |M_R(H)^{\sigma}| = |M_R(H)/ \Gamma_k|.
\end{align*}

\subsection{Asymptotic formula for $N(k,R,H,B)$}\label{asympt n(k,R,H,B)}

There is an Euler product decomposition of the global Fourier transform, given by \[ \hat{f}_{H,R}(x;s)= \prod_v\hat{f}_{H,R,v}(x_v;s),\] and by Lemma \ref{restricted ram lft}, this can be expanded as a Dirichlet series \begin{equation}\label{dir sris eq}
    \hat{f}_{H,R}(x;s)= \sum_{n \geq 1}\frac{\alpha_n(G,x)}{n^s}.
\end{equation} Furthermore, each local Fourier transform can be written as \[\hat{f}_{H,R,v}(x_v;s) = 1+ \lambda_x(v) q_v^{-\frac{s}{a_R(H)}} +  O_{\epsilon}\left(q_v^{-\left(\frac{1}{a_R(H)} + \epsilon \right)s}\right) \]  where $\lambda_x$ is defined as in (\ref{frob fn 2}). Therefore, by \cite[Prop. $2.2$]{Alberts2021HarmonicAA} the global Fourier transform $\hat{f}_{H,R}(x;s)$ satisfies that $\hat{f}_{H,R}(x;s)= \zeta_k(a_R(H)^{-1}s)^{b_R(H,x)}G(x;s)$ for some $G(x;s)$ which is holomorphic in a region of the form $\re(s) > a_R(H) - \frac{c}{\log(|\im(s)|+3)}$ for some constant $0<c < \frac{1}{4}$. We will make use of the following Tauberian theorem, found in \cite[Thm.~III]{Delange1954} or \cite[Cor, p.~$121$]{narkiewicz1983number}. \begin{thm} \label{taub}
        Let $f(s) = \sum_{n \geq 1}\beta_n n^{-s}$ with $\beta_n \geq 0$ be a convergent Dirichlet series for $\re(s) > a> 0$, and assume it may be written as \[f(s) = g(s)(s-a)^{-w} +h(s)\] where $g(a) \neq 0$, and $g(s),h(s)$ holomorphic on $\re(s) \geq a$, $w>0$. Then \[\sum_{n \leq B} \beta_n = \frac{g(a)}{a(w-1)!}B^a(\log B)^{w-1} + o(B^a(\log B)^{w-1}).\]
    \end{thm}

 \begin{lem} \label{asympt form for dir coeffs}
    Let $x \in \mathcal{O}_S^{\times} \otimes G^{\wedge}$ and let $\alpha_n(G,x)$ be the Dirichlet coefficient from (\ref{dir sris eq}). Then \[ \sum_{n \leq B} \alpha_n(G,x) = c(k,R,H,x) B^{a_R(H)}(\log B)^{b_R(H,x)-1}(1 + o(1)) \] where \begin{equation} \label{lc dir coeff}
         c(k,R,H,x) = \frac{a_R(H)^{b_R(H,x)-1}}{(b_R(H,x)-1)!} \lim_{s \rightarrow a_R(H)}(s - a_R(H))^{b_R(H,x)} \hat{f}_{H,R}(x;s).
    \end{equation}
\end{lem}

\begin{proof}
    We have that $\hat{f}_{H,R}(x;s)$ may be written as \[\hat{f}_{H,R}(x;s)= \zeta_k(a_R(H)^{-1}s)^{b_R(H,x)}G(x;s)\] where the product $\zeta_k(a_R(H)^{-1}s)^{b_R(H,x)}G(x;s)$ is equal to \begin{align*}
        \zeta_k(a_R(H)^{-1}s)^{b_R(H,x)}G(x;s) = (s-& a_R(H))^{-b_R(H,x)} a_R(H)^{b_R(H,x)} \\ & \times  [(a_R(H)^{-1}s-1) \zeta_k(a_R(H)^{-1}s)]^{b_R(H,x)}G(x;s).
    \end{align*} In particular, $\hat{f}_{H,R}(x;s)$ has the form \[\hat{f}_{H,R}(x;s)= (s - a_R(H))^{-b_R(H,x)}g(s)\] where $g(s)$ is given by \[g(s) := a_R(H)^{b_R(H,x)}[(a_R(H)^{-1}s-1) \zeta_k(a_R(H)^{-1}s)]^{b_R(H,x)}G(x;s)\] and $G(x;s)$ is holomorphic for $\re(s) \geq a_R(H)^{-1}$.  Then applying Theorem \ref{taub} to the Dirichlet coefficients of $\hat{f}_{H,R}(x;s)$, we obtain the asymptotic formula.
\end{proof} \noindent To relate this back to the original height zeta function, we will use the following minor variant of an abstract version of Poisson summation, due to Frei, Loughran and Newton in \cite[Prop.~$3.9$]{Frei_2022}. \begin{prop}[Poisson Summation] \label{poi lem eq}
    Let $S$ be a finite set of places containing the archimedean places, those dividing $|G|$ and such that $\mathcal{O}_S$ has trivial class group and $x \in \mathcal{O}_S^{\times} \otimes G^{\wedge}$. For $\re(s) > a_R(H)$ the global Fourier transform $\hat{f}_{H,R}(x;s)$ exists and is holomorphic in this region. Furthermore the following Poisson summation formula holds:
    \begin{equation} \label{poi eq}
        \sum_{\chi \in \Hom( \AAA^{\times}/k^{\times}, G)}\frac{f_{R}(\chi)}{H(\chi)^s} = \frac{1}{|\mathcal{O}_k^{\times} \otimes G^{\wedge}|}\sum_{x \in \mathcal{O}_S^{\times} \otimes G^{\wedge}}\hat{f}_{H,R}(x;s) \end{equation}
      \end{prop}
\begin{proof}
    The proof of \cite[Prop.~$3.9$]{Frei_2022} uses the conductor rather than any height function, however the proof in our case is very similar and hence omitted.
\end{proof} 

\noindent We may now state and prove the following result. If $v$ is a non-archimedean place, then $\zeta_{k,v}(s)$ denotes the Euler factor of $\zeta_{k}(s)$ at $v$ and if $v$ is archimedean $\zeta_{k,v}(s) = 1$.

\begin{thm}\label{balanced ht restr ram thm}
    Let $H$ be a big balanced height with associated weight function $w$, $S$ be a finite set of places containing the infinite places, those dividing $|G|$ and such that $\mathcal{O}_S$ has trivial class group. Denote by $S_f$ the set of finite places of $S$. Let $R \subseteq G(-1)^*$ be a non-empty Galois-stable subset, and $M_R(H) = \{\gamma \in R : w(\gamma) \:\text{is minimal}\}$. Then the counting function \[N(k,R,H,B) = \# \{\varphi \in \Hom (\Gamma_k, G) : \:H(\varphi)\leq B, \: \rho_{G,v}(\varphi_v) \in  R \cup \{0\} \: \forall v \not\in S\}\] satisfies the asymptotic formula \[
        N(k,R,H,B) \sim  c_{k,R,G,H}B^{a_R(H)} (\log B)^{b_R(H) - 1},
    \] where \[
            a_R(H) = \left(\min_{\gamma \in R} w(\gamma)\right)^{-1}  \quad \text{and} \quad
           b_R(H) = |M_R(H)/\Gamma_k|,
        \] 
        and

         \begin{align}\label{lc c(k,R,G,H) euler}
           &  c_{k,R,G,H} = \frac{a_R(H)^{b_R(H)-1}(\Res_{s=1}\zeta_k(s))^{b_R(H)}}{(b_R(H)-1)! |\mathcal{O}_k^{\times} \otimes G^{\wedge}||G|^{|S_f|}}
             \sum_{x \in \mathcal{X}(k,R,H)} \\ \nonumber
             &  \left(\prod_{v \not\in S} \sum_{\substack{\chi_v \in \Hom(\mathcal{O}_v^{\times},G) \\ \rho_{G,v}(\chi_v) \in R \cup \{0\}}} \frac{\langle \chi_v , x_v \rangle}{H_v(\chi_v)^{a_R(H)}\zeta_{k,v}(1)^{b_R(H)}} \prod_{v \in S}\sum_{\chi_v \in \Hom(k_v^{\times},G)}\frac{\langle \chi_v , x_v \rangle}{H_v(\chi_v)^{a_R(H)}\zeta_{k,v}(1)^{b_R(H)}} \right)
                        \end{align}
    \noindent and $\mathcal{X}(k,R,H)$ is given by \[ \mathcal{X}(k,R,H) =  \left\{ x \in k^{\times} \otimes G^{\wedge}: \begin{array}{ll}
         &\text{For all but finitely many $v$ such that}  \\
         & \rho_{G,v}(\chi_v) \in M_R(H) \cup \{0\}\: \text{we have} \: \langle \chi_v , x_v \rangle = 1
    \end{array}
    \right\}.\]

\end{thm}
\noindent In order for the leading constant to be well-defined, the set $\mathcal{X}(k,R,H)$ needs to be finite. The proof of this may be found in Proposition \ref{identif}. The balancedness assumption on the height function comes into play here as it is necessary for $H$ to be balanced with respect to $R$ for this set to be finite. 
\begin{proof}
    From Lemma \ref{asympt form for dir coeffs}, we have an asymptotic formula for $\sum_{n \leq B}\alpha_n(G,x)$. It also follows from the definition that $b_R(H,x) \leq b_R(H)$ for all $x \in\mathcal{O}_S^{\times} \otimes G^{\wedge}$. We may apply the Poisson summation formula from Proposition \ref{poi lem eq} to obtain the following asymptotic for the coefficients $f_n$ of the height zeta function: \[ \sum_{n \leq B}f_n \sim  c_{k,R,G,H}B^{a_R(H)}(\log B)^{b_R(H)-1}(1+o(1))\] where the leading constant is given by \[ c_{k,R,G,H} = \frac{a_R(H)^{b_R(H)-1}}{(b_R(H)-1)! |\mathcal{O}_k^{\times} \otimes G^{\wedge}|}\sum_{\substack{x \in \mathcal{O}_S^{\times} \otimes G^{\wedge} \\ b_R(H,x) = b_R(H) }}\lim_{s \rightarrow a_R(H)}(s-a_R(H))^{b_R(H)}\hat{f}_{H,R}(x;s).\] The $x \in \mathcal{O}_S^{\times} \otimes G^{\wedge}$ such that $b_R(H,x) = b_R(H)$ are exactly the $x$ in the set $\mathcal{X}(k,R,H)$. Furthermore, by \cite[Prop.~$2.3$]{Frei_2022} the limit appearing in the leading constant satisfies \[\lim_{s \rightarrow a_R(H)}(s-a_R(H))^{b_R(H)}\hat{f}_{H,R}(x;s) = (\Res_{s=1} \zeta_k (s))^{b_R(H)} \prod_{v}\frac{\hat{f}_{H,R,v}(x_v;a_R(H))}{\zeta_{k,v}(1)^{b_R(H)}}.\] We consider the places $v \not\in S$ and $v \in S$ separately. We have \[\prod_{v \not\in S}\frac{\hat{f}_{H,R,v}(x_v;a_R(H))}{\zeta_{k,v}(1)^{b_R(H)}} = \prod_{v \not\in S} \sum_{\substack{\chi_v \in \Hom(\mathcal{O}_v^{\times}.G) \\ \rho_{G,v}(\chi_v) \in R \cup \{0\}}} \frac{\langle \chi_v , x_v \rangle }{H_v(\chi_v)^{a_R(H)}}\zeta_{k,v}(1)^{-b_R(H)}\] and \[\prod_{v \in S}\frac{\hat{f}_{H,R,v}(x_v;a_R(H))}{\zeta_{k,v}(1)^{b_R(H)}} = \frac{1}{|G|^{|S_f|}}\prod_{v \in S}\sum_{\chi_v \in \Hom(k_v^{\times},G)}\frac{\langle \chi_v , x_v \rangle }{H_v(\chi_v)^{a_R(H)}}\zeta_{k,v}(1)^{-b_R(H)}.\] In particular, the limit equals \begin{align*}
        & \lim_{s \rightarrow a_R(H)}(s-a_R(H))^{b_R(H)}\hat{f}_{H,R}(x;s) = \frac{(\Res_{s=1} \zeta_k (s))^{b_R(H)} }{|G|^{|S_f|}} \\
        & \times  \prod_{v \not\in S} \sum_{\substack{\chi_v \in \Hom(\mathcal{O}_v^{\times},G) \\ \rho_{G,v}(\chi_v) \in R \cup \{0\}}} \frac{\langle \chi_v , x_v \rangle }{H_v(\chi_v)^{a_R(H)}\zeta_{k,v}(1)^{b_R(H)}}  \prod_{v \in S}\sum_{\chi_v \in \Hom(k_v^{\times},G)}\frac{\langle \chi_v , x_v \rangle }{H_v(\chi_v)^{a_R(H)}\zeta_{k,v}(1)^{b_R(H)}},
    \end{align*} and the explicit formula follows from taking the sum over $x \in \mathcal{X}(k,R,H)$.\end{proof}
\noindent Theorem \ref{balanced ht restr ram thm} proves Theorem \ref{main thm abt restr ram} in the balanced height case. As a formal consequence of Theorem \ref{balanced ht restr ram thm}, by setting $R=G(-1)^*$ we obtain an asymptotic formula for the number of all $G$-extensions of bounded balanced height, which proves Corollary \ref{formal consequence} in the balanced case, and the leading constant for Corollary \ref{formal consequence} in this case is given by \begin{align}\label{lc c(k,G,H) euler}
           &  c_{k,G,H} = \frac{a(H)^{b(H)-1}(\Res_{s=1}\zeta_k(s))^{b(H)}}{(b(H)-1)! |\mathcal{O}_k^{\times} \otimes G^{\wedge}||G|^{|S_f|}}
             \sum_{x \in \mathcal{X}(k,G,H)} \\ \nonumber
             &  \left(\prod_{v \not\in S} \sum_{\chi_v \in \Hom(\mathcal{O}_v^{\times},G)} \frac{\langle \chi_v , x_v \rangle}{H_v(\chi_v)^{a(H)}\zeta_{k,v}(1)^{b(H)}} \prod_{v \in S}\sum_{\chi_v \in \Hom(k_v^{\times},G)}\frac{\langle \chi_v , x_v \rangle}{H_v(\chi_v)^{a(H)}\zeta_{k,v}(1)^{b(H)}} \right).
                        \end{align}

We have the following corollary to Theorem \ref{balanced ht restr ram thm} which we will use to handle general heights.

 \begin{cor}\label{corollary unif upper bound}
 Let $R\subseteq G(-1)^*$ and $S$ be as in Theorem \ref{balanced ht restr ram thm} and $H$ be a balanced height function with respect to $R$ with associated weight function $w$. Write $H_{\min} = \min_{\varphi \in \Hom(\Gamma_k, G)}H(\varphi)$ For fixed $G$ and $k$, there exist a constant $C_{k,R,G,H_{\min}}$  (different to $ c_{k,R,G,H}$) such that the upper bound 
     \[\# \left\{ \varphi \in \Hom (\Gamma_k, G) : \begin{array}{ll}
          &  \rho_{G,v}(\varphi_v) \in R \cup \{0\} \forall v \not\in S \\
          &  H(\varphi) \leq B
     \end{array}   \right\}  \\\leq C_{k,R,G,H_{\min}} B^{a_R(H)}(\log B)^{b_R(H) - 1}\] holds.
 \end{cor}
\begin{proof}
    The counting function $N(k,R,H,B)$ is a step function.  When $B$ is large enough, the bound holds by Theorem \ref{balanced ht restr ram thm}. For small $B$ consider the minimum value of the height function. In our case we have that $H_{\min}$ is bounded below by some constant $C_{H_{\min}}$. If $B < C_{H_{\min}}\leq H_{\min}$, the counting function is clearly equal to zero, and the bound holds. Then we may choose $C_{k,R,G,H_{\min}}$ such that the bound holds for all $B$.
\end{proof}

\section{Counting by general heights}

\subsection{The Greenberg-Wiles formula}

Let $G$ be a finite abelian group and $M$ be a finite $\Gamma_k$-module. Given a $G/M$-extension $\psi$ of a number field $k$, we wish to show that if there exists a lift of $\psi$ to a $G$-extension of $k$ then there exists a lift only ramified at a set of places $S_0$ and those ramifying in $\psi$. We do this by considering the cohomology groups associated to the generalised Selmer groups, and applying a result due to Wiles \cite[Prop~$1.6$]{andrewwiles} known as the Greenberg-Wiles formula. \begin{defi}
    Let $k$ be a number field and $M$ a finite $\Gamma_k$-module. A \emph{collection of local conditions} is a family $\mathcal{L} = \{\mathcal{L}_v\}$ where each $\mathcal{L}_v \subset H^{1}(k_v, M) $ and for all but finitely many places $v$ we have \[\mathcal{L}_v = H_{\unr}^{1}(k_{v}, M) = \Ker\left(H^{1}(k_v,M) \rightarrow H^{1}(k_v^{\unr},M) \right),\] where $k_v^{\unr}$ is the maximal unramified extension of $k_v$. Furthermore, the \emph{dual local conditions} $\mathcal{L}^{*}$ of $\mathcal{L}$ is the collection  $\mathcal{L}^{*} = \{\mathcal{L}_v^{\perp}\} $  where each $\mathcal{L}_v^{\perp} \subset H^{1}(k_v, M^{*})$ is the orthogonal complement to $\mathcal{L}_v$ with respect to local Tate duality and $M^{*}$ denotes the Cartier dual of $M$.
\end{defi}
 \noindent We define the cohomology groups corresponding to generalised Selmer groups as \[ H_{\mathcal{L}}^{1}(k, M) = \{ x \in H^{1}(k, M): \: \text{res}_v(x) \in \mathcal{L}_v, \: \forall v \}\] and  \[ H_{\mathcal{L}^{*}}^{1}(k, M^{*}) = \{ x \in H^{1}(k, M^{*}): \: \text{res}_v(x) \in \mathcal{L}_v^{\perp}, \: \forall v \}.\] 

\begin{prop}\label{existence of a lift}
    Let $k$ be a number field and $A$ be a finite abelian group. Then there exists a finite set of places $S_0$ such that the map \[H^{1}(k, A) \rightarrow \bigoplus_{v \not\in S_0}\frac{H^{1}(k_v,A)}{H_{\unr}^{1}(k_{v}, A)}\] is surjective.
\end{prop} \noindent This has been shown for $A = \ZZ/p\ZZ$ for a prime $p$ by Koymans and Pagano in \cite[Prop.~$6.2$]{koymanspagano}
\begin{proof}
    Since $A$ is a finite abelian group, we have $A \cong \bigoplus_{\ell \: \text{prime}} \ZZ/\ell^r \ZZ$. Furthermore  $H^{1}(k, A)\cong \bigoplus_{\ell \: \text{prime}} H^{1}(k,\ZZ/\ell^r \ZZ)$, and the same holds for the local cohomology groups for $v \not\in S_0$ (including the unramified cohomology groups). It then follows by an application of B\'ezout's identity that we can reduce to the case where $A = \ZZ/\ell^r \ZZ$ for a prime $\ell$ and integer $r$.
Let $S_0$ be a finite set of places, $n = \ell^r$ and $\mathcal{L} = \{\mathcal{L}_v\}$ be the collection of local conditions given by $\mathcal{L}_v = H^1_{\unr}(k_v, \ZZ/n\ZZ)$ for $v \not\in S_0$ and $\mathcal{L}_v = H^1(k_v, \ZZ/n\ZZ)$ for $v\in S_0$. We will assume $S_0$ contains all $v\mid n\cdot \infty$ and enlarge $S_0$ throughout the proof.

Furthermore, for all finite $v_0 \in S_0$, let $\mathcal{L}^{(v_0)} = \{\mathcal{L}_v^{(v_0)}\}$ be given by $\mathcal{L}_v^{(v_0)} = \mathcal{L}_v$ for all $v \neq v_0$ and $\mathcal{L}_v^{(v_0)}= H^1(k_{v_0}, \ZZ/n\ZZ)$. Consider the sequence \[0 \rightarrow  H_{\mathcal{L}}^{1}(k, \ZZ/n\ZZ) \rightarrow  H_{\mathcal{L}^{(v_0)}}^{1}(k, \ZZ/n\ZZ) \xrightarrow{\eta_{v_0}} \frac{H^{1}(k_{v_0},\ZZ/n\ZZ)}{H_{\unr}^{1}(k_{v_0}, \ZZ/n\ZZ)}.\] This is exact, as the kernel of $\eta_{v_0}$ is made up of those elements of $H_{\mathcal{L}^{(v_0)}}^{1}(k, \ZZ/n\ZZ)$ which map to $H_{\unr}^{1}(k_{v_0}, \ZZ/n\ZZ)$, and this is $H_{\mathcal{L}}^{1}(k, \ZZ/n\ZZ)$. On the other hand, the first map is the inclusion map, and its image is thus $H_{\mathcal{L}}^{1}(k, \ZZ/n\ZZ)$. Proving surjectivity of the map \[H^{1}(k, A) \rightarrow \bigoplus_{v \not\in S_0}\frac{H^{1}(k_v,A)}{H_{\unr}^{1}(k_{v}, A)}\] is equivalent to obtaining all elements of the form $(0,\cdots,0,f,0,\cdots,0)$ for all $f$ in the $v_0$-th component of the sum, for every place $v_0 \not\in S_0$, as these elements generate the infinite sum. It is enough then to prove that $\eta_{v_0}$ is surjective for large enough $S_0$, as getting all elements $(0,\cdots,0,f,0,\cdots,0)$ at a fixed component is equivalent to the surjectivity of $\eta_{v_0}$. We will show that \[\frac{\# H_{\mathcal{L}^{(v_0)}}^{1}(k, \ZZ/n\ZZ)}{\# H_{\mathcal{L}}^{1}(k, \ZZ/n\ZZ)} = \frac{\# H^{1}(k_{v_0},\ZZ/n\ZZ)}{\# H_{\unr}^{1}(k_{v_0}, \ZZ/n\ZZ)}.\] We apply the Greenberg-Wiles formula in the form of \cite[Thm.~$8.7.9$]{neukirch2013cohomology} twice, first to the conditions $\mathcal{L}$ and then to $\mathcal{L}^{(v_0)}$. Taking quotients, we obtain \[\frac{\# H_{\mathcal{L}^{(v_0)}}^{1}(k, \ZZ/n\ZZ) \# H_{\mathcal{L}^*}^{1}(k, \mu_n)}{\# H_{\mathcal{L}^{(v_0)*}}^{1}(k, \mu_n)\# H_{\mathcal{L}}^{1}(k, \ZZ/n\ZZ)} = \frac{\# H^{1}(k_{v_0},\ZZ/n\ZZ)}{\# H_{\unr}^{1}(k_{v_0}, \ZZ/n\ZZ)}.\] By \cite[Thm.~$7.2.15$]{neukirch2013cohomology}, for finite places $v \nmid n$ the dual of $ H_{\unr}^{1}(k_{v}, \ZZ/n\ZZ)$ under local Tate duality is given by $ H_{\unr}^{1}(k_{v}, \mu_n)$ and so for $v\not\in S_0$ the dual local conditions $\mathcal{L}^* = \{\mathcal{L}_v^{\perp}\}$ are given by $\mathcal{L}_v^* = H_{\unr}^{1}(k, \mu_n)$. For $v \in S_0$ they are given by $\mathcal{L}_v^* = 0$. As the local conditions $\mathcal{L}^{(v_0)*}$ are more restrictive, we have the containment $H_{\mathcal{L}^{(v_0)*}}^{1}(k, \mu_n) \subseteq H_{\mathcal{L}^*}^{1}(k, \mu_n)$. The group $H_{\mathcal{L}^*}^{1}(k, \mu_n)$ is given by \[
   H_{\mathcal{L}^*}^{1}(k, \mu_n) = \{x \in  k^{\times}/k^{\times n}:  v(x) \equiv 0 \bmod n \: \forall 
 v \not\in S_0, \:  x \in k_v^{\times n} \: \forall v \in S_0\}.
\] Consider the Tate-Shafarevich group \[\Sha^1(k,\mu_n) = \Ker \left( k^{\times}/k^{\times n} \rightarrow \prod_v k_v^{\times}/k_v^{\times n} \right). \] Then $\Sha^1(k,\mu_n) \subseteq H_{\mathcal{L}^*}^{1}(k, \mu_n)$ and $H_{\mathcal{L}^*}^{1}(k, \mu_n)$ is contained in the Selmer group \[\Sel^n(k) = \{ x \in k^{\times}/k^{\times n} :  v(x) \equiv 0 \bmod n \: \forall 
 v\}.\] Let $x \in \Sel^n(k) \backslash  \Sha^1(k,\mu_n)$. Then there exists some place $v_x$ such that $x \notin k_{v_x}^{\times n}$. If this place is contained in $S_0$, we must have $x \not\in H_{\mathcal{L}^*}^{1}(k, \mu_n)$. In particular, the Selmer group is finite and for all $x \in \Sel^n(k) \backslash  \Sha^1(k,\mu_n)$ there is some $v_x$ such that $x \notin k_{v_x}^{\times n}$. If we enlarge $S_0$ to ensure it contains $ \{v \mid n \cdot \infty \} \cup \{v_x : x \in \Sel^n(k)\}$, then we have that $H_{\mathcal{L}^*}^{1}(k, \mu_n) \subseteq \Sha^1(k,\mu_n)$ and hence $H_{\mathcal{L}^*}^{1}(k, \mu_n) = \Sha^1(k,\mu_n)$. We then get that $H_{\mathcal{L}^*}^{1}(k, \mu_n) \subseteq H_{\mathcal{L}^{(v_0)*}}^{1}(k, \mu_n)$. It follows from this that the map $\eta_{v_0}$ is surjective. 
\end{proof}

\subsection{Dominated Convergence}

So far Theorem \ref{balanced ht restr ram thm} holds for balanced height functions, but Theorem $1.1$ does not impose any balancednes conditions on the height function. The case where the height function is unbalanced is more delicate as there may be infinite sums appearing in the leading constant. 

\begin{thm}\label{restr. ram. gen hts}
    Let $H$ be a (possibly unbalanced) big height function, $S$ a finite set of places containing the infinite places, those dividing $|G|$ and such that $\mathcal{O}_S$ has trivial class group and $R \subseteq G(-1)^*$ be a non-empty, Galois-stable subset. Then \[\# \left\{ \varphi \in \Hom (\Gamma_k, G): \begin{array}{ll}
         & \rho_{G,v}(\varphi_v) \in R \cup \{0\} \:\ \forall v \not\in S,  \\
         & H(\varphi) \leq B
    \end{array}
    \right\} \sim c_{k,R,G,H}B^{a_R(H)}(\log B)^{b_R(H)-1}\] where the leading constant is given by the sum \begin{equation}\label{lc unbalanced g ext restr ram}
    c_{k,R,G,H} = \sum_{\psi \in \Hom(\Gamma_k, G/\langle M_R(H) \rangle)}c_{k,R,\langle M_R(H) \rangle_{\psi},H}
    \end{equation} where $a_R(H)$, $b_R(H)$ and each $c_{k,R,\langle M_R(H) \rangle_{\psi},H}$ are as in Theorem \ref{balanced ht restr ram thm}, and the sum is over the $G/\langle M_R(H) \rangle$-extensions of $k$ which embed into a $G$-extension of $k$ with restricted ramification imposed by $R$.
\end{thm}
\noindent To prove this result we will employ an argument due to Koymans and Rome in Step $1$ of the proof of \cite[Thm.~$1.1$]{koymansdct}. A difference perspective is offered by Alberts, Lemke Oliver, Wang and Wood via their inductive methods framework \cite[Thm.~$2.1$]{alberts2025inductivemethodscountingnumber}, however here we use the Koymans-Rome approach. Let $\pi : G \rightarrow G/\langle M_R(H) \rangle$ be the natural quotient map, $\psi \in \Hom(\Gamma_k, G/\langle M_R(H) \rangle)$ and define \[N_{\psi}(k,R,H,B) = \# \left\{ 
\varphi \in \Hom (\Gamma_k, G) : \begin{array}{ll}
     & \rho_{G,v}(\varphi_v) \in R \cup \{0\} \:\ \forall v \not\in S, \\
     & H(\varphi) \leq B, \: \pi \circ \varphi = \psi
\end{array} \right\}\]  and the aim is to use the dominated convergence theorem to show that \[\sum_{\psi }\lim_{B \rightarrow \infty }\frac{N_{\psi}(k,R,H,B)}{B^{a_R(H)} (\log B)^{b_R(H) -1}} = \lim_{B \rightarrow \infty } \sum_{\psi }\frac{N_{\psi}(k,R,H,B)}{B^{a_R(H)} (\log B)^{b_R(H) -1}}.\] This in turn gives that the sum of leading constants \[ c_{k,R,G,H} = \sum_{\psi \in \Hom(\Gamma_k, G/\langle M_R(H) \rangle)}c_{k,R,\langle M_R(H) \rangle_{\psi},H}\] converges. In order to apply the dominated convergence theorem we first find a uniform upper bound for the quantity \begin{equation}\label{N'_i(X)}
    \frac{N_{\psi}(k,R,H,B)}{B^{a_R(H)} (\log B)^{b_R(H) -1}}.
\end{equation}

\noindent Consider the numerator of (\ref{N'_i(X)}). If $\psi : \Gamma_k \rightarrow G/\langle M_R(H) \rangle$ does not admit a lift to $G$ whose ramification type is restricted by $R$, then \[N_{\psi}(k,R,H,B) = \# \left\{ 
\varphi \in \Hom (\Gamma_k, G) : \begin{array}{ll}
     & \rho_{G,v}(\varphi_v) \in R \cup \{0\} \:\ \forall v \not\in S, \\
     & H(\varphi) \leq B, \: \pi \circ \varphi = \psi
\end{array} \right\} = 0\] thus we may assume that a lift exists. Recall that $f_R(\varphi) = 1$ if $\rho_{G,v}(\varphi_v) \in R \cup \{0\}$ for all $v \not\in S$. By Proposition \ref{existence of a lift}, we may choose a lift $\tilde{\psi} : \Gamma_k \rightarrow G$ such that $\tilde{\psi}$ is only ramified at the set of places $S_0$ from Proposition \ref{existence of a lift} and the places which ramify in $\psi$ and such that $f_R(\tilde{\psi})=1$. Moreover, every other lift of $\psi$ is given by $\tilde{\psi}+t$ for $t: \Gamma_k \rightarrow \langle M_R(H) \rangle$, so that the lifts are parametrised by homomorphisms $t : \Gamma_k \rightarrow \langle M_R(H) \rangle$. We may rewrite (\ref{N'_i(X)}) as \[\frac{\# \{t :\Gamma_k \rightarrow \langle M_R(H) \rangle : \tilde{\psi} + t \: \text{surj}., \: H(\tilde{\psi} + t) \leq B, \: f_R(\tilde{\psi} + t) = 1 \}}{B^{a_R(H)}(\log B)^{b_R(H)-1}}.\] If $T$ is a set of places of $k$ containing the wild places, we denote by $H_T(\tilde{\psi} + t)$ the height \[H_T(\tilde{\psi} + t) = \prod_{v \not\in T}H_v(\tilde{\psi} + t).\] Furthermore, we will write $S_H$ for the finite set of places such that the formula $H_v(\tilde{\psi} + t) = q_v^{w(\rho_{G,v}(\tilde{\psi} + t))}$ does not hold and $S_{H,0} = S_H \cup S_0$.
 \begin{lem} \label{remove surj}
 For all finite sets of places $T$ containing $S_{H,0}$, we have the following upper bound: \begin{align*}
    \# \{t: \Gamma_k \rightarrow \langle M_R(H) \rangle : \: & \tilde{\psi} + t \: \text{surj.}, \: H_T(\tilde{\psi} + t) \leq X\} \\
    & \ll_{k,G} |\langle M_R(H) \rangle|^{  |T|}\#  \{t: \Gamma_k \rightarrow \langle M_R(H) \rangle :  H_{S_{H,0}}(\tilde{\psi} + t) \leq X \}.
\end{align*} 
\end{lem}

\begin{proof}
    Let $ H_T(\tilde{\psi} + t) \leq X$. By Proposition \ref{existence of a lift} we may choose $t'$ that is only ramified at $T$ but such that $\tilde{\psi} + t + t'$ is unramified at $T \backslash S_{H,0}$. In particular $\rho_{G,v}(\tilde{\psi} + t + t')$ is trivial for $v \in T \backslash S_{H,0}$. Hence \[H_{S_{H,0}}(\tilde{\psi} + t + t') = \prod_{v \not\in S_{H,0}}H_v(\tilde{\psi} + t + t') = \prod_{v \not\in S_{H,0}}q_v^{w(\rho_{G,v}(\tilde{\psi} + t + t'))}.\] This final product can be written as \[\prod_{v \not\in T}q_v^{w(\rho_{G,v}(\tilde{\psi} + t + t'))}\prod_{v \in T \backslash S_{H,0}}q_v^{w(\rho_{G,v}(\tilde{\psi} + t + t'))} \] and the latter of the two products is trivial as $\tilde{\psi} + t + t'$ is unramified for $v \in T \backslash S_{H,0}$. But $t'$ is unramified for $v \not\in T$ and thus the ramification type evaluated at $t'$ is trivial at these places. Hence we get that $H_T(\tilde{\psi} + t) = H_{S_{H,0}}(\tilde{\psi} + t + t')$. Furthermore there are $ \ll_{k,G}|\langle M_R(H) \rangle|^{|T|}$ choices for $t'$. 
\end{proof}

\begin{lem}\label{unif upper bound}
    Let $\Phi(\psi)$ be the norm of the conductor of the homomorphism $\psi: \Gamma_k \rightarrow G/\langle M_R(H) \rangle$. There exists some $\epsilon > 0$ that depends at most on $k,R,G$ and $H$ such that \[\frac{N_{\psi}(k,R,H,B)}{B^{a_R(H)}(\log B)^{b_R(H)-1}} \ll_{k,R,G,H} \frac{|\langle M_R(H) \rangle|^{[k:\QQ]\omega(\Phi(\psi))}}{\prod_{v \mid \Phi(\psi)}q_v^{1 + \epsilon}}\] where $q_v$ is the cardinality of the residue field at $v$ and the implied constant does not depend on $\psi$.
\end{lem} \noindent Note also that this bound does not depend on $\tilde{\psi}$ or $t$.
\begin{proof}
    
 Set $S_{H,0}(\psi)$ to be the union of the set of places $S_0$ from Proposition \ref{existence of a lift} and the finite set of bad places $S_H$ for the height $H$ joined with the ramification locus of $\psi$. We will bound the counting function $N_{\psi}(k,R,H,B)$ in the numerator of (\ref{N'_i(X)}) using Lemma \ref{remove surj} and Corollary \ref{corollary unif upper bound}. Let $\Phi_{S_{H,0}}(\psi) = \prod_{v \not\in S_{H,0}}\Phi_v(\psi)$. If $v \mid \Phi_{S_{H,0}}(\psi)$, then it is a place outside $S_{H,0}$ in the ramification locus of $\psi$. At these places $\rho_{G,v}(\psi) \in (G/\langle M_R(H) \rangle)(-1)$ is non-trivial and thus satisfies that $w(\rho_{G,v}(\psi)) > a_R(H)^{-1}$. In particular, by our choice of lift, $\tilde{\psi}$ ramifies at the same places as $\psi$ and thus we have for all places $v \mid \Phi_{S_{H,0}}(\psi)$ that \begin{equation}\label{ramify in psi}
w(\rho_{G,v}(\tilde{\psi} + t)) = w(\rho_{G,v}(\tilde{\psi}))w(\rho_{G,v}(t))  \geq (1+\epsilon)a_R(H)^{-1}.\end{equation} On the other hand if $v \in \Omega_k \backslash S_{H,0}$ ramifies in $t: \Gamma_k \rightarrow \langle M_R(H) \rangle$ but not in $\psi$, we have $w(\rho_{G,v}(\tilde{\psi}))$ is trivial and $w(\rho_{G,v}(t)) = a_R(H)^{-1}$, so \begin{equation}\label{ramify in t}
 w(\rho_{G,v}(\tilde{\psi} + t)) = w(\rho_{G,v}(\tilde{\psi}))w(\rho_{G,v}(t)) = a_R(H)^{-1}.\end{equation} Using (\ref{ramify in psi}), and the fact that the lifts of $\psi$ are parametrised by homomorphisms $t: \Gamma_k \rightarrow \langle M_R(H) \rangle$ we may bound the numerator of (\ref{N'_i(X)}) by \[\# \left\{t: \Gamma_k \rightarrow \langle M_R(H) \rangle : \tilde{\psi} + t \: \text{surj}., \: H_{S_{H,0}(\psi)}(\tilde{\psi} + t) \ll_{k,G} \frac{B}{\prod_{v | \Phi_{S_{H,0}}(\psi)}q_v^{(1+\epsilon)a_R(H)^{-1}}}\right\}.\] This is now in the correct form to apply Lemma \ref{remove surj}. By setting $T = S_{H,0}(\psi)$ and $X = \frac{B}{\prod_{v | \Phi_{S_{H,0}}(\psi)}q_v^{(1+\epsilon)a_R(H)^{-1}}}$ in the statement of Lemma \ref{remove surj}, we obtain the upper bound for (\ref{N'_i(X)}) given by
\begin{align*}
   & \ll_{k,R,G,H} \\
   & \frac{|\langle M_R(H) \rangle|^{|S_{H,0}(\psi)|} \# \left\{ t:\Gamma_k \rightarrow \langle M_R(H) \rangle : \begin{array}{ll}
      & H_{S_{H,0}}(\tilde{\psi} + t) \leq  \frac{B}{\prod_{v | \Phi_{S_{H,0}}(\psi)}q_v^{(1+\epsilon)a_R(H)^{-1}}}, \\
      & f_R(\tilde{\psi} + t) = 1
 \end{array} \right\}}{B^{a_R(H)}(\log B)^{b_R(H) -1}},
\end{align*}
 allowing us to remove the surjectivity condition on $\tilde{\psi} + t$. It follows from (\ref{ramify in psi}) and (\ref{ramify in t}) that  $H(t) ^{|G/\langle M_R(H) \rangle|} \ll_{k,G,H} H_{S_{H,0}}(t) ^{|G/\langle M_R(H) \rangle|} \leq H_{S_{H,0}}(\tilde{\psi} + t)$. Let $R' = R \cap \langle M_R(H) \rangle$. Now we wish to upper bound \[\# \left\{ t:\Gamma_k \rightarrow \langle M_R(H) \rangle : f_{R'}(t) = 1, \: H(t)^{|G/\langle M_R(H) \rangle|} \leq  \frac{B}{\prod_{v \mid \Phi_{S_{H,0}}(\psi)}q_v^{(1+\epsilon)a_R(H)^{-1}}} \right\}.\] Clearly as $\langle M_R(H) \rangle$ is generated by $M_R(H)$, $H$ is balanced with respect to $R$ when restricted to $\langle M_R(H) \rangle$. By applying Corollary \ref{corollary unif upper bound} with $G = \langle M_R(H) \rangle$ and $R = R'$, we have the upper bound \begin{multline*}
     \# \left\{ t:\Gamma_k \rightarrow \langle M_R(H) \rangle :f_{R'}(t) = 1, \: H(t)^{|G/\langle M_R(H) \rangle|} \leq  \frac{B}{\prod_{v \mid \Phi_{S_{H,0}}(\psi)}q_v^{(1+\epsilon)a_R(H)^{-1}}} \right\} \\
     \ll_{k,R,G,H}\left(\frac{B^{a_R(H)}}{\prod_{v | \Phi_{S_{H,0}}(\psi)}q_v^{(1+\epsilon)}}\right)(\log B)^{b_R(H)-1},
 \end{multline*} as the non-surjective homomorphisms $t : \Gamma_k \rightarrow \langle M_R(H) \rangle$ are negligible. We also have the bounds $|S_{H,0}(\psi)| \ll_{k,G}[k:\QQ] \omega(\Phi(\psi))$ and $\Phi(\psi) \ll_{k,G} \Phi_{S_{H,0}}(\psi)$. Putting everything together we obtain the upper bound \begin{equation*}\label{unif up bnd}
     \frac{N_{\psi}(k,R,H,B)}{B^{a_R(H)}(\log B)^{b_R(H) -1}} \ll_{k,G, R,H} \frac{|\langle M_R(H) \rangle|^{[k:\QQ]\omega(\Phi(\psi))}}{\prod_{v | \Phi(\psi)}q_v^{1+\epsilon}}. \qedhere \end{equation*}\end{proof}

 \begin{proof}[Proof of Theorem \ref{restr. ram. gen hts}]

 The uniform upper bound in Lemma \ref{unif upper bound} holds for each $\psi$, and \[\frac{|\langle M_R(H) \rangle|^{[k:\QQ]\omega(\Phi(\psi))}}{\prod_{v \mid \Phi(\psi)}q_v^{1+\epsilon}} \ll_{k,G}\frac{1}{\Phi(\psi)^{1+ \delta}}\] for some $\delta >0$ and thus \[\sum_{\psi \in \Hom(\Gamma_k, G/\langle M_R(H) \rangle)} \frac{|\langle M_R(H) \rangle|^{[k:\QQ]\omega(\Phi(\psi))}}{\prod_{v | \Phi(\psi)}q_v^{1+\epsilon}} \] converges. To apply dominated convergence we need to show that that \[\lim_{B \rightarrow \infty}\frac{N_{\psi}(k,R,H,B)}{B^{a_R(H)} (\log B)^{b_R(H) -1}}\] exists and is finite. We know that $H$ is balanced with respect to $R$ when restricted to the $\langle M_R(H) \rangle$-extensions and since $G$ is a constant abelian group, the $\langle M_R(H) \rangle$-extensions correspond to the collection of $G$-extensions realising a given $G/\langle M_R(H) \rangle$-extension. Thus, by Theorem \ref{balanced ht restr ram thm}, we have that as $B\rightarrow \infty$, \[\frac{N_{\psi}(k,R,H,B)}{B^{a_R(H)} (\log B)^{b_R(H) -1}} \sim c_{k,R,\langle M_R(H) \rangle_{\psi},H},\] and each inner twist $\langle M_R(H) \rangle_{\psi}$ is trivial as $\langle M_R(H) \rangle$ is abelian. By dominated convergence, the counting function $N(k,R,H,B) = \sum_{\psi}N_{\psi}(k,R,H,B)$ satisfies the asymptotic formula in Theorem \ref{restr. ram. gen hts} and the leading constant $c_{k,R,G,H}$ is given by the convergent sum (\ref{lc unbalanced g ext restr ram}).\end{proof} \noindent In the case $R = G(-1)^*$, we have proven \cite[Conj.~$9.6$]{loughransantens} for finite abelian groups $G$.

\section{Interpretation of results via stacks}\label{inter resulst via stacks}

There is a natural formulation of Theorem \ref{balanced ht restr ram thm} and Theorem \ref{restr. ram. gen hts} in terms of the stack $BG$ and the Batyrev-Manin conjecture on stacks. Using this viewpoint we can obtain a formula for the leading constant in terms of Tamagawa measures and Brauer groups. These formulas may be used to prove results concerning the equidistribution of rational points on $BG$, corresponding to what is informally known as the Malle-Bhargava heuristics in the Malle conjecture literature, although with potential Brauer-Manin obstruction. In addition to this, we can use this version of the leading constant to show that the existence of number fields with restricted ramification type is controlled by a Brauer-Manin obstruction on $BG$.

\subsection{Brauer groups of stacks and the Brauer-Manin pairing}

We begin by considering the stack $BG$ for a finite abelian group $G$. By \cite[Lem.~$2.1$]{loughransantens}, for a finite abelian group $G$, the groupoid $BG(k)$ corresponds to the groupoid of homomorphisms $\Gamma_k \rightarrow G$ where the isomorphisms are given by conjugation in $G$ (in our case conjugation is trivial as $G$ is abelian). We write $BG[k]$ for the set of isomorphism classes of $BG(k)$. In this way the problem of counting homomorphisms $\Gamma_k \rightarrow G$ from Theorem \ref{main thm abt restr ram} is equivalent to counting elements of $BG[k]$.

The Brauer-Manin pairing we wish to use comes from the partially unramified Brauer group. We first define the partial adelic space, which is the natural space in order to ensure that the Brauer-Manin pairing is well-defined for the elements of the partially unramified Brauer group appearing in the leading constant. \begin{defi}[Partial adelic space]
    Let $C \subset G(-1)^{*}$ be a Galois-stable subset and $\rho_{G,v}$ be the ramification type. For a tame non-archimedean place $v$ we define  \[BG(\mathcal{O}_v)_{C} = \{ \varphi_v \in BG(k_v) : \rho_{G,v}(\varphi_v) \in C \cup \{0\}\}.\] Then the \emph{partial adelic space with respect to C} is the limit over all finite sets of places $S$ given by \[BG(\AAA_k)_C = \lim_S \prod_{v \in S}BG(k_v) \prod_{v \not\in S}BG(\mathcal{O}_v)_C.\]
\end{defi}

In particular, the set $BG(\mathcal{O}_v)_{C}$ is the set of those homomorphisms $\varphi_v :\Gamma_{k_v} \rightarrow G$ such that the image of $\varphi_v$ under the ramification type is either trivial or in $C$. An element of $BG(\mathcal{O}_v)_{C}$ is called a \emph{partial $v$-adic integral point with respect to $C$}.

We also need the following \emph{partially unramified Brauer group} of $BG$. Let $C \subset G(-1)^{*}$. From \cite[Thm.~$7.4$]{loughransantens}, an element $b \in \br BG$ is in the group $\br_{C}BG$ if and only if $b$ evaluates trivially on $BG(\mathcal{O}_v)_{C}$ for all but finitely many places $v$. Moreover, by \cite[Cor.~$6.30$]{loughransantens} if the elements of $C$ generate $G$, the group $\br_C BG / \br k$ is finite. Using this partially unramified Brauer group, we have the following partially unramified Brauer-Manin pairing from \cite[Lem.~$7.5$]{loughransantens} \[ \br_C BG \times BG(\AAA_k)_C \rightarrow \QQ/\ZZ, \] which is well-defined and continuous.
 \noindent It is the partially unramified Brauer group in this pairing that will appear in our leading constant later on.

\subsection{Tamagawa measures}\label{tamagawa measure sec}

We define a Tamagawa measure for heights on $BG$ analogously to that used by Peyre in \cite{peyretamagawa}. This was first defined by Loughran and Santens in \cite[$\S 8$]{loughransantens}. We start with local Tamagawa measures.

\begin{defi}[Local Tamagawa measures] \label{loc tam measure}
  Let $v$ be a place of $k$ and $W_v \subseteq BG[k_v]$ be a subset. Then the local Tamagawa measure associated to the choice of height $H = (H_v)_{v \in \Omega_k}$ is defined to be \[\tau_{v,H_v}(W_v) = \sum_{\varphi_v \in [W_v]} \frac{1}{|\aut(\varphi_v)|H_v(\varphi_v)^{a(H)}},\] where $a(H) = (\min_{\gamma \in G(-1)^*}w(\gamma))^{-1}$ and $w$ is the weight function corresponding to $H$.
\end{defi} \noindent This sum is finite as $\text{char}(k_v) = 0$ and we have $\aut(\varphi_v) \cong G$ for all $\varphi_v \in BG[k_v]$, so this is a well-defined measure on the set $BG[k_v]$ of isomorphism classes of $k_v$ points of $BG(k_v)$. 

The global Tamagawa measures are the product of the local Tamagawa measures, and we ensure this product is convergent by introducing convergence factors. In the leading constant (\ref{lc c(k,R,G,H) euler}) we have the convergence factors $\zeta_{k,v}(1)^{-b_R(H)}$. We can take the corresponding global Tamagawa measure to be the measure from \cite[Lem.~$8.19$]{loughransantens}, that is, \begin{equation} \label{tamagawa to use}
\tau_{H} = (\Res_{s = 1}\zeta_{k}(s))^{b_R(H)}\prod_{v| \infty}\tau_{v,H_v}\prod_{v \nmid \infty}(1 - 1/q_v)^{b_R(H)}\tau_{v,H_v}.\end{equation} 

\subsection{The partially unramified Brauer group}
Recall the following definition of the set \[ \mathcal{X}(k,R,H) =  \left\{ x \in k^{\times} \otimes G^{\wedge}: \begin{array}{ll}
         &\text{For all but finitely many $v$ such that}  \\
         & \rho_{G,v}(\chi_v) \in M_R(H) \cup \{0\}\: \text{we have} \: \langle \chi_v , x_v \rangle = 1
    \end{array}
    \right\}\] from Theorem \ref{balanced ht restr ram thm}.

\begin{prop} \label{identif}

 Let $\varphi \in BG(\AAA_k)_{M_R(H)}$ and $H$ be balanced with respect to $R$. Then the set $\mathcal{X}(k, R, H)$ is finite and we have the equality \[\sum_{x \in \mathcal{X}(k, R, H)} \prod_v \langle \chi_v,x_v \rangle \quad = \sum_{b \in \br_{M_R(H)}BG/\br k} e^{2 \pi i \langle b , \varphi \rangle_{BM} }\] where $\chi \in \Hom(\AAA^{\times}/k^{\times},G)$ is the character corresponding to $\varphi$ under the global Artin map and $\langle \cdot , \cdot \rangle$ is the Pontryagin pairing.
\end{prop}

\begin{proof}

Let $b \in \br BG$. By \cite[Thm.~$7.4$]{loughransantens} we have $b \in \br_{M_R(H)}BG$ if and only if it evaluates trivially on $BG(\mathcal{O}_v)_{M_R(H)}$ for all but finitely many places $v$. From \cite[Lem.~$10.23$]{loughransantens}, there is a canonical isomorphism $H^{1}(k, \widehat{G}) \cong k^{\times} \otimes G^{\sim}$, where $G^{\sim}$ is the $\QQ/\ZZ$-dual of $G$, which is isomorphic to the Pontryagin dual $G^{\wedge}$ of $G$ via the map $\QQ/\ZZ \rightarrow S^1$, $t \mapsto e^{2 \pi i t}$. By \cite[Lem.~$6.38$]{loughransantens}, since $M_R(H)$ generates $G$ and $G$ is abelian, there are no transcendental Brauer group elements. With respect to the isomorphism $H^{1}(k, \widehat{G}) \cong k^{\times} \otimes G^{\sim}$, there is an equality between $\mathcal{X}(k, R, H)$ and $\br_{M_R(H)}BG /\br k$. Moreover if $H$ is balanced we have that $M_R(H)$ generates $G$. The finiteness of $\mathcal{X}(k,R,H)$ then follows from the equality between $\mathcal{X}(k, R, H)$ and $\br_{M_R(H)}BG /\br k$ and \cite[Cor.~$6.30$]{loughransantens}.

Using local Tate duality, the Pontryagin pairing is identified with the cup product, and this in turn is identified with the Brauer-Manin pairing via \cite[Lem.~$6.4$]{loughransantens}. Furthermore, by local class field theory we have that for all but finitely many tame places $v$, the $\varphi_v \in BG(\mathcal{O}_v)_{M_R(H)}$ correspond to the $\chi_v \in \Hom(k_v^{\times}, G)$ such that $\rho_{G}(\chi_v) \in M_R(H) \cup \{0\}$. 
\end{proof}

\subsection{Counting number fields via the stack $BG$}

We will rewrite the leading constant $c_{k,R,G,H}$ from Theorem \ref{balanced ht restr ram thm} for a height $H$ that is balanced with respect to $R$ in terms of Tamagawa measures and Brauer groups. We will use the following lemma concerning the Tamagawa measure on the partial adelic space and sums of Euler products, and the map $\QQ/\ZZ \rightarrow S^{1} : t \mapsto e^{2 \pi i t}$.

\begin{lem} \label{lem 7.17}
    For each $b \in \br_{M_R(H)}BG/\br k$, consider the Euler product 
    \begin{align*}
        \hat{\tau}_H(b) & : = \int_{BG(\AAA_k)_{M_R(H)}} e^{2 \pi i \langle b, \varphi \rangle_{BM} } d \tau_{H}(\varphi) \\
        & = (\Res_{s=1}\zeta_k(s))^{b_R(H)} \prod_v\zeta_{k,v}(1)^{-b_R(H)} \int_{BG(k_v)} e^{2 \pi i \inv_vb(\varphi_v)} d \tau_{v,H_v}(\varphi_v).
    \end{align*} Then \[ |\br_{M_R(H)}BG / \br k| \tau_{H}(BG(\AAA_k)_{M_R(H)}^{\br}) = \sum_{b \in \br_{M_R(H)}BG / \br k}\hat{\tau}_H(b)\] is a finite sum of Euler products.
\end{lem} 

\begin{proof}
    See \cite[Lem.~$8.21$]{loughransantens}.
\end{proof}

\begin{lem}\label{lc tamagawa measures}
    Let $S$ be the set of places from Theorem \ref{balanced ht restr ram thm} and let $W_{R,S} = \prod_{v \in S}BG(k_v) \allowbreak \prod_{v \not\in S}BG(\mathcal{O}_v)_{R}$. The leading constant $c_{k,R,G,H}$ from Theorem \ref{balanced ht restr ram thm} is equal to \begin{equation}\label{lc tamagawa n(k,r,g,b)}
         \frac{ |G|a_R(H)^{b_R(H)-1}|\br_{M_R(H)}BG/\br k|\tau_H(W_{R,S} \cap BG(\AAA_k)_{M_R(H)}^{\br})}{|\widehat{G}(k)| (b_R(H)-1)!},\end{equation} where $\widehat{G}$ is the Cartier dual of $G$.
\end{lem}

\begin{proof}
    We start with the leading constant from Theorem \ref{balanced ht restr ram thm} in the form 

\begin{align*}
           &  c_{k,R,G,H} = \frac{a_R(H)^{b_R(H)-1}(\Res_{s=1}\zeta_k(s))^{b_R(H)}}{(b_R(H)-1)! |\mathcal{O}_k^{\times} \otimes G^{\wedge}||G|^{|S_f|}}
             \sum_{x \in \mathcal{X}(k,R,H)} \\
             &  \left(\prod_{v \not\in S} \sum_{\substack{\chi_v \in \Hom(\mathcal{O}_v^{\times},G) \\ \rho_{G,v}(\chi_v) \in R \cup \{0\}}} \frac{\langle \chi_v , x_v \rangle}{H_v(\chi_v)^{a_R(H)}\zeta_{k,v}(1)^{b_R(H)}} \prod_{v \in S}\sum_{\chi_v \in \Hom(k_v^{\times},G)}\frac{\langle \chi_v , x_v \rangle}{H_v(\chi_v)^{a_R(H)}\zeta_{k,v}(1)^{b_R(H)}} \right).
                        \end{align*}
We use local class field theory to identify $\Hom(k_v^{\times}, G)$ and $ \Hom(\Gamma_{k_v},G)$ and this last group is equal to $BG(k_v)$. For $v \not\in S$, we have the equality \[\sum_{\substack{\chi_v \in \Hom(\mathcal{O}_v^{\times},G) \\ \rho_{G,v}(\chi_v) \in R \cup \{0\}}} \frac{\langle \chi_v , x_v \rangle}{H_v(\chi_v)^{a_R(H)}} = \frac{1}{|G|}\sum_{\chi_v \in \Hom(k_v^{\times},G)}\frac{f_{R,v}(\chi_v)\langle \chi_v , x_v \rangle}{H_v(\chi_v)^{a_R(H)}}\] where the function $f_{R,v}$ is the indicator function of the condition $\rho_{G,v}(\varphi_v) \in R \cup \{0\}$ for all $v \not\in S$. In terms of stacks, we may view $f_{R,v}$ as the indicator function of the set $W_{R,S} = \prod_{v \in S}BG(k_v) \prod_{v \not\in S}BG(\mathcal{O}_v)_{R}$. By \cite[Lem.~$10.22$]{loughransantens} we have that \[\frac{1}{|\mathcal{O}_k^{\times} \otimes G^{\wedge}||G|^{|S_f|}} = \frac{|G|}{|\widehat{G}(k)||G|^{|S|}}\] where $S_{f}$ denotes the set of infinite places in $S$. We may thus write the leading constant as \[\frac{|G|a_R(H)^{b_R(H)-1}(\Res_{s=1}\zeta_k(s))^{b_R(H)}}{(b_R(H)-1)! |\widehat{G}(k)|}
            \sum_{x \in \mathcal{X}(k,R,H)} \prod_{v} \frac{1}{|G|} \sum_{\chi_v \in \Hom(k_v^{\times},G)} \frac{f_{R,v}(\chi_v)\langle \chi_v , x_v \rangle}{H_v(\chi_v)^{a_R(H)}\zeta_{k,v}(1)^{b_R(H)}}.\]  By Proposition \ref{identif}, there is an equality of sums \[\sum_{x \in \mathcal{X}(k, R, H)} \prod_v \langle \chi_v,x_v \rangle \quad = \sum_{b \in \br_{M_R(H)}BG/\br k} e^{2 \pi i \langle b , \varphi \rangle_{BM}},\] for $\varphi \in W_{R,S}$ and the latter sum is non-zero and equal to $|\br_{M_R(H)}BG/\br k|$ for $\varphi \in W_{R,S} \cap BG(\AAA_k)_{M_R(H)}^{\br}$. Furthermore for every place $v$ we have \[\frac{1}{|G|}\sum_{\varphi_v \in BG(k_v)}\frac{f_{R,v}(\varphi_v)}{H_v(\varphi_v)^{a_R(H)}}=\sum_{\varphi_v \in W_{R,S,v}}\frac{1}{|\aut(\varphi_v)|H_v(\varphi_v)^{a_R(H)}} = \tau_{v,H_v}(W_{R,S,v}),\] and using the formula for the global Tamagawa measure (\ref{tamagawa to use}) we have \[\tau_{H}(W_{R,S}) = (\Res_{s=1}\zeta_k(s))^{b_R(H)}\prod_{v\mid \infty} \tau_{v,H_v}(W_{R,S,v}) \prod_{v \nmid \infty}(1 -1/q_v)^{b_R(H)} \tau_{v,H_v}(W_{R,S,v}). \qedhere\]

\end{proof}
\noindent We may now reformulate Theorem \ref{balanced ht restr ram thm} using the language of stacks.
\begin{thm}[Counting $G$-extensions of bounded balanced height, stacky version]\label{stacky balanced hts} Let $H$ be a big balanced height with respect to a non-empty Galois-stable subset $R \subseteq G(-1)^*$, and $S$ the set of places from Theorem \ref{balanced ht restr ram thm}. Let $W_{R,S} \subset BG(\AAA_k)_{R}$ be $W_{R,S} = \prod_{v \in S}BG(k_v) \prod_{v \not\in S}BG(\mathcal{O}_v)_{R}$. Then we have 
\[\frac{1}{|G|}\# \{ \varphi \in BG[k] :\varphi \in W_{R,S}, H(\varphi) \leq B\} \sim c(k,R,G,H)B^{a_R(H)}(\log B)^{b_R(H)-1}\] where $a_R(H)$ and $b_R(H)$ are as in Theorem \ref{balanced ht restr ram thm} and when $H$ is a balanced height with respect to $R$ the leading constant $c(k,R,G,H)$ is given by \[ \frac{a_R(H)^{b_R(H)-1}|\br_{M_R(H)}BG/\br k|\tau_H(W_{R,S} \cap BG(\AAA_k)_{M_R(H)}^{\br}) }{|\widehat{G}(k)| (b_R(H)-1)!}.\]

\end{thm}
\noindent We write $\varphi \in W_{R,S}$ to mean that the image of $\varphi$ under the map $BG[k] \rightarrow BG[\AAA_k]_{R}$ is in $W_{R,S}$. Note that we have the equality $c_{k,R,G,H} = |G|c(k,R,G,H)$.

We may also write our results for general height functions in terms of stacks but more care must be taken. Consider the homomorphism $BG \rightarrow B(G/\langle M_R(H) \rangle)$, called the Iitaka fibration. By \cite[Lem.~$3.31$]{loughransantens}, we have that $H$ is balanced with respect to $R$ when restricted to the fibres of this Iitaka fibration, and these fibres are given by $B\langle M_R(H) \rangle_{\psi}$. Here $\langle M_R(H) \rangle_{\psi}$ is the inner twist of $\langle M_R(H) \rangle$ by a lift of $\psi \in B(G/\langle M_R(H) \rangle)[k]$ along $BG[k] \rightarrow B(G/\langle M_R(H) \rangle)[k]$ (which in the case of a finite abelian group is trivial), providing such a lift exists. Thus, $H$ is balanced with respect to $R$ when restricted to the stack $B\langle M_R(H) \rangle$. The stacky version of counting $G$-extensions of bounded general height follows immediately from Theorems \ref{restr. ram. gen hts} and \ref{stacky balanced hts}. 
 
\noindent We will now explicitly calculate the example appearing in the introduction.
\subsection{Proof of Example $1.8$}\label{proof of ex}
Notice that $G = \ZZ/4\ZZ$ is invariant under the exponentiation action of $\hat{\ZZ}^{\times}$. Hence, via the isomorphism from \cite[Lem.~$3.3$]{loughransantens} we may work with $G$ rather than $G(-1)$ and in particular, we can view the weight function as a function defined on the elements of $G$.\\
\textbf{Part $(1)$}. We have $M_G(H) = \{1,3\}$ and this set generates $G$, so it follows that $H$ is a balanced height function. We also have $a(H) = b(H) = 1$, and we write $\br := \br_{e,\{1,3\}}B(\ZZ/4\ZZ) = \br_{\{1,3\}}B(\ZZ/4\ZZ) \cap \br_e B\ZZ/4\ZZ$ where $ \br_eB\ZZ/4\ZZ = \{b \in \br B\ZZ/4\ZZ : b(e) = 0\}$. Thus by Theorem \ref{stacky balanced hts}, we have that \[ \frac{1}{4}\# \{\varphi \in B\ZZ/4\ZZ[\QQ] : H(\varphi) \leq B, \: \varphi \: \text{is completely split at $2$ and $\infty$}\}\sim c(\QQ,\ZZ/4\ZZ,H)B\] where the leading constant is of the form \[c(\QQ,\ZZ/4\ZZ,H) = \frac{a(H)^{b(H)-1}|\br_{e,\{1,3\}}B(\ZZ/4\ZZ)|\tau_H(B\ZZ/4\ZZ(\AAA_{\QQ})_{\{1,3\}}^{\br}) }{|\widehat{\ZZ/4\ZZ}(\QQ)| (b(H)-1)!}.\] 
We have $|\widehat{\ZZ/4\ZZ}(\QQ)| = 2$ and $a(H)^{b(H)-1}=1$ and thus the effective cone constant is $\frac{1}{2}$, and $(b(H)-1)! = 1$. Then to calculate the leading constant we calculate the partially unramified Brauer group $\br_{e,\{1,3\}}B(\ZZ/4\ZZ)$. \begin{lem}\label{unr br grp ex}
    The Brauer group $\br_{e,\{1,3\}}B(\ZZ/4\ZZ)$ is of order $2$ and the non-trivial element corresponds to $-4 \in \br_e B\ZZ/4\ZZ$.
\end{lem}

\begin{proof} We write $\br_e B\ZZ/4\ZZ \cong \{1, -4\}$ and we use the identification $\br_e B\ZZ/4\ZZ = H^{1}(\QQ, \mu_4)$ from \cite[Lem.~$6.2$]{loughransantens}. By \cite[Lem.~$6.27$]{loughransantens} the residue at $\{1,3\} \in \ZZ/4\ZZ$ is the restriction map $H^{1}(\QQ, \mu_4) \rightarrow H^{1}(\QQ(i),\mu_4)$ and by \cite[Thm.~$6.29(3)$]{loughransantens} the relevant Brauer group is the kernel of this map. This is the group $H^1(\Gal(\QQ(i)/\QQ), \mu_4)$ and in our case by \cite[Prop.~$9.1.6$]{neukirch2013cohomology} this is exactly $\ZZ/2\ZZ$. Via Kummer theory this corresponds to those elements of $\QQ^{\times}/\QQ^{\times 4}$ which are $4$th powers in $\QQ(i)$. Indeed, the element $-4$ satisfies $-4 = ( 1 + i)^{4}$, and this non-trivial element generates the Brauer group.
\end{proof}
 \noindent Therefore by a minor variant of Lemma \ref{lem 7.17}, the leading constant equals \[c(\QQ,\ZZ/4\ZZ,H) = \frac{1}{2}  |\br_{e,\{1,3\}}B(\ZZ/4\ZZ)| \tau_H(B\ZZ/4\ZZ(\AAA_{\QQ})_{\{1,3\}}^{\br}) = \frac{1}{2} (\hat{\tau}_H(1) + \hat{\tau}_H(-4)),\] where each $\hat{\tau}_H(b)$ is given by \[(\Res_{s=1}\zeta_{\QQ}(s))^{b(H)} \prod_v \zeta_{v}(1)^{-b(H)} \int_{BG(\QQ_v)} e^{2 \pi i \inv_vb(\varphi_v)} d \tau_{v,H_v}(\varphi_v),\] and the product is taken over all the places $v$ of $\QQ$. 
We now compute the local densities at each place of $\QQ$.
\begin{lem} \label{local dens R=G} Let $\varphi \in B\ZZ/4\ZZ[\QQ]$ be completely split at $2$ and $\infty$. Then at $v = \infty$ and $v = 2$ we have the local densities $\hat{\tau}_{v,H_v}(1) = \hat{\tau}_{v,H_v}(-4) = \frac{1}{4}$, and at the odd primes we have that \[\hat{\tau}_{v,H_v}(1) = \begin{cases}
        1 + \frac{2}{p} + \frac{1}{p^2} & p \equiv 1 \bmod 4,  \\
        1 +  \frac{1}{p^2} & p \equiv 3 \bmod 4 
    \end{cases}\] and \[\hat{\tau}_{v,H_v}(-4) = \begin{cases}
        1 + \frac{2}{p} + \frac{1}{p^2} & p \equiv 1 \bmod 4,  \\
        1 -  \frac{1}{p^2} & p \equiv 3 \bmod 4 
    \end{cases}.\] 
\end{lem}
\begin{proof}
     We consider the factors coming from primes $p \neq 2$, the infinite place and $2$ separately. By assumption we are only considering extensions completely split at $2$ and $\infty$. Then at these places the local cocycles are trivial. The element $-4 \in \br_{e,\{1,3\}}B\ZZ/4\ZZ$ therefore evaluates trivially at these homomorphisms, and the local densities at these places are equal to $|\aut(1)| = \frac{1}{4}$. 
     
     For the primes $p \neq 2$, we separate into the cases $p \equiv 1 \bmod 4$ and $p \equiv 3 \bmod 4$, and use the mass formulae in \cite[Cor.~$8.11$]{loughransantens} and \cite[Thm.~$8.23$]{loughransantens}. At primes $p \equiv 1 \bmod 4$, the element $-4 = (1 + i )^4$ is a $4$th power in $\QQ_p(i)$, and thus evaluates trivially on $B\ZZ/4\ZZ(\QQ_p)$. In particular, at these primes, by \cite[Cor.~$8.11$]{loughransantens} we have 
     \begin{align*}
         \hat{\tau}_{p,H_p}(1) = \hat{\tau}_{p,H_p}(-4) = & \sum_{\varphi_p \in B\ZZ/4\ZZ(\QQ_p)}\frac{1}{|\aut(\varphi_p)|H_p(\varphi_p)} \\
    = & \sum_{c \in \{0,1,2,3\}}\frac{1}{p^{w(c)}} = 1 + \frac{2}{p} + \frac{1}{p^2},
     \end{align*} as all elements of $\mu_4$ are Galois-invariant for $p \equiv 1 \bmod 4$. 

At primes $p \equiv 3 \bmod 4$, the trivial element of the relevant Brauer group once again evaluates trivially on $B\ZZ/4\ZZ(\QQ_p)$. At these primes, the elements $\{1,3\}$ are no longer invariant under the Galois action, and thus by \cite[Cor.~$8.11$]{loughransantens} we have \[
   \hat{\tau}_{p,H_p}(1) = \tau_{p,H_p}(B\ZZ/4\ZZ(\QQ_p)) = \sum_{\varphi_p \in B\ZZ/4\ZZ(\QQ_p)}\frac{1}{|\aut(\varphi_p)|H_p(\varphi_p)} \\
   = \sum_{c \in \{0,2\}}\frac{1}{p^{w(c)}} = 1 + \frac{1}{p^2}.
\]

On the other hand, at $-4 \in \br_{e,\{1,3\}}B\ZZ/4\ZZ$, for primes  $p \equiv 3 \bmod 4$ we have \begin{align*}
      \hat{\tau}_{p,H_p}(-4) = & \int_{B\ZZ/4\ZZ(\QQ_p)}e^{2 \pi i \inv_p(-4(\varphi_p))}\tau_{p,H_p}(\varphi_p) \\
    = & \sum_{\varphi_p \in B\ZZ/4\ZZ(\QQ_p)}\frac{e^{2 \pi i \inv_p(-4(\varphi_p))}}{|\aut(\varphi_p)|H_p(\varphi_p)} = \sum_{c \in \{0,2\}} \frac{\chi_p(c)}{p^{w(c)}},
\end{align*} where we have applied \cite[Thm.~$8.23$]{loughransantens} with $f = 1$. The Galois character $\chi_p(c)$ may be viewed as a mod $p$ Dirichlet character via class field theory. Then $\chi_p(c) = 1$ when $c=0$ and $\chi_p(c) = \left( \frac{-4}{p} \right)$ when $c=2$, and since $p \equiv 3 \bmod 4$, this is equal to $-1$. In particular in this case \[\hat{\tau}_{p,H_p}(-4) = 1 - \frac{1}{p^2}. \qedhere\]
\end{proof} \noindent We have $(\Res_{s=1}\zeta_{\QQ}(s)) = 1$. The leading constant for Part $(1)$ follows from the calculation of the convergence factors at each place, which are given by $1$ at $v=\infty$, $\zeta_2(1)^{-1} = \frac{1}{2}$ at $p = 2$ and $\zeta_p(1)^{-1} = \left(1-\frac{1}{p}\right)$ at all other primes $p$.

\noindent It is possible to explicitly describe the Brauer-Manin obstruction occurring in this example.

\begin{lem}
    Let $G = \ZZ/4\ZZ$ and $k = \QQ$, and let $S$ be a finite set of primes such that $2, \infty \not\in S$. Then a $G$-extension $\varphi$ such that \begin{itemize}
        \item $2$ and $\infty$ are completely split in $\varphi$,
        \item $\rho_{G,p}(\varphi) \in \{1,3\}$ for all $p \not\in S$,
        \item $\rho_{G,p}(\varphi) \in \{2\}$ for all $p \in S$.
    \end{itemize}  exists if and only if the number of primes $p \in S$ such that $p \equiv 3 \bmod 4$ is even.
\end{lem}

\begin{proof}
    Suppose there exists a $\varphi \in B\ZZ/4\ZZ[\QQ] \subseteq B\ZZ/4\ZZ[\AAA_{\QQ}]_{\{1,3\}}^{\br}$ satisfying the above properties. Let $-4 \in \br_{e,\{1,3\}}B(\ZZ/4\ZZ)$ be the non-trivial element. By the global reciprocity law, we have $\sum_{p}\inv_p(-4(\varphi))=0$. For $p = 2,\infty$, we have $\inv_p(-4(\varphi))=0$ as we only consider extensions completely split at $2$ and $\infty$, so $-4$ evaluates trivially at these places. For primes $p \not\in S$, we have $\varphi_p \in B\ZZ/4\ZZ(\ZZ_p)_{\{1,3\}}$. Thus $\inv_p(-4(\varphi))=0$ as $-4$ so is unramified at $\{1,3\}$. For $p \in S$ such that $p \equiv 1 \bmod 4$, the Brauer group element is trivial and we have $\inv_p(-4(\varphi))=0$. On the other hand for $p \equiv 3 \bmod 4$ we have that $\inv_p(-4(\varphi)) = \frac{1}{2}$, as $-4$ is of order $2$. Therefore it follows that there is an even number of such $p$ in the finite set $S$. 
    
    For the reverse implication, consider the subset of $B\ZZ/4\ZZ[\AAA_{\QQ}]_{\{1,3\}}$ given by $W = \prod_{v \mid 2\cdot \infty}\{\mathbbm{1}_v\}\prod_{p \in S}B\ZZ/4\ZZ(\ZZ_p)_{\{2\}}\prod_{p \not\in S}B\ZZ/4\ZZ(\ZZ_p)_{\{1,3\}}$, where $\mathbbm{1}_v$ is the trivial homomorphism as $2$ and $\infty$ are completely split. Since $\# \{p \in S: p \equiv 3 \bmod 4\}$ is even, for $(\varphi_p)_p \in W$ we have $\sum_{p}\inv_p(-4(\varphi))=0$. In particular, $W \subseteq B\ZZ/4\ZZ[\AAA_{\QQ}]_{\{1,3\}}^{\br}$. By Theorem \ref{grunwald}, the image of the map \[B\ZZ/4\ZZ[\QQ] \rightarrow B\ZZ/4\ZZ[\AAA_{\QQ}]_{\{1,3\}}^{\br}\] is dense, and so there exists some $\varphi$ close to $(\varphi_p)_p$.
\end{proof}

\textbf{Part} $(2)$. We have $R = \{1,3\}$ and we count all $\ZZ/4\ZZ$-extensions of $\QQ$ completely split at $2$ and $\infty$ whose ramification type lies in $R$. In other words, we wish to count all elements of the set $W = \prod_{v \mid 2\cdot \infty}B\ZZ/4\ZZ(\QQ_v)\prod_{p \neq 2}B\ZZ/4\ZZ(\ZZ_p)_{\{1,3\}}$.

In this case $M_R(H) = \{1,3\}$ and these elements generate $G$, thus $H$ is a balanced height function with respect to $R$. Moreover the partially unramified Brauer group with respect to $R$ is equal to the Brauer group in Lemma \ref{unr br grp ex}, and in particular has order $2$ and is generated by $-4$. By Theorem \ref{stacky balanced hts}, the leading constant $c(\QQ,\{1,3\}, \ZZ/4\ZZ,H)$ for the asymptotic formula when $R = \{1,3\}$ is given by \[\frac{a_R(H)^{b_R(H)-1} |\br_{e,\{1,3\}}B(\ZZ/4\ZZ)|\tau_H( W \cap B\ZZ/4\ZZ(\AAA_{\QQ})_{M_R(H)}^{\br})}{|\widehat{\ZZ/4\ZZ}(\QQ)| (b_R(H)-1)!},\] and we once again have $a_R(H) = b_R(H) = 1$ and the effective cone constant is equal to $\frac{1}{2}$.

\begin{lem}  Let $\varphi \in B\ZZ/4\ZZ[\QQ]$ be completely split at $2$ and $\infty$. Then the local densities at $v = \infty$ and $v = 2$ are given by $\hat{\tau}_{v,H_{v}}(1)= \hat{\tau}_{v,H_{v}}(-4) = \frac{1}{4}$ and at the primes we have \[\hat{\tau}_{p,H_p}(b) = \begin{cases}
        1 + \frac{2}{p} & p \equiv 1 \bmod 4,\\
        \frac{1}{4} & p \equiv 3 \bmod 4,
    \end{cases}\] for all $b \in \br_{e,\{1,3\}}B(\ZZ/4\ZZ)$.
\end{lem}

\begin{proof}
    The cases for $\infty$ and $p=2$ are the same as in Lemma \ref{local dens R=G}. For primes $p \neq 2$ let $f_{R_p}$ be the indicator function for $B\ZZ/4\ZZ(\ZZ_p)_{\{1,3\}}$. For primes $p \equiv 1 \bmod 4$, we have that $\inv_pb(\varphi_p)$ is trivial for all $b \in \br_{e,\{1,3\}}B(\ZZ/4\ZZ)$ and hence \[\int_{B\ZZ/4\ZZ(\QQ_p)}f_{R_p}(\varphi_p)  e^{2 \pi i \inv_pb(\varphi_p)} d \tau_{p,H_p}(\varphi_p) = \tau_{p,H_p}(B\ZZ/4\ZZ(\ZZ_p)_{\{1,3\}}) = 1 + \frac{2}{p}, \] as by \cite[Cor.~$8.11$]{loughransantens} we have $\tau_{p,H_p}(B\ZZ/4\ZZ(\ZZ_p)_{\{1,3\}}) = 1 + \frac{\# M_R(H)^{\Gamma_{\QQ_p}}}{p}$. For primes $p \equiv 3 \bmod 4$ we apply \cite[Thm.~$8.23$]{loughransantens} with $f = \mathbbm{1}_{\{1,3\}}$, the indicator function for the condition that $\rho_{G,p}(\varphi_p) \in \{0,1,3\}$, to obtain \[\hat{\tau}_{p,H_p}(b) = \sum_{c \in \{0,2\}}\frac{\mathbbm{1}_{\{1,3\}}(c)\chi_p(c)}{p^{w(c)}}.\] However $\mathbbm{1}_{\{1,3\}}(2) = 0$, and thus each local density is equal to $1$.
\end{proof} 

\noindent The leading constant is then given by \begin{align*}
     2 \times & \frac{1}{64}\prod_{\substack{p \: \text{prime}  \\ p \equiv 1 \bmod 4}}\left(1-\frac{1}{p}\right)\left(1+\frac{2}{p}\right)\prod_{\substack{p \: \text{prime}  \\ p \equiv 3 \bmod 4}}\left(1-\frac{1}{p}\right) \\
             = & \frac{1}{32} \prod_{\substack{p \: \text{prime}  \\ p \equiv 1 \bmod 4}}\left(1-\frac{1}{p}\right)\left(1+\frac{2}{p}\right)\prod_{\substack{p \: \text{prime}  \\ p \equiv 3 \bmod 4}}\left(1-\frac{1}{p}\right).
\end{align*}

\textbf{Part} $(3)$. Consider the set $R = \{2\}$. Then $M_R(H) = \{2\}$, and as this does not generate $G = \ZZ/4\ZZ$, the height function $H$ is not balanced with respect to $R$. Thus we cannot apply Theorem \ref{stacky main thm intro} directly in this case.

Instead, one must consider the subgroup $\langle M_R(H) \rangle = \ZZ/2\ZZ$ and the corresponding Iitaka fibration $B\ZZ/4\ZZ \rightarrow B\ZZ/2\ZZ$, which sends each $\ZZ/4\ZZ$-extension of $\QQ$ with restricted ramification imposed by $R$ to its unique quadratic subfield. We then count the rational points on each fibre of the Iitaka fibration, as $H$ is now balanced when restricted to each fibre, and sum over all such fibres.

 \begin{lem}
    Let $K = \QQ(\sqrt{d}) /\QQ \in B\ZZ/2\ZZ(\QQ)$ be such that $d$ is the sum of two squares. Then the fibre of this point along Iitaka fibration is $B \ZZ/2\ZZ$, and if $H_K$ is the restriction of the height $H$ to this fibre then $a_R(H_K) = b_R(H_K) = 1$.
\end{lem}

\begin{proof}
    If $K = \QQ(\sqrt{d})$ satisfies that $d$ is the sum of two squares then it embeds into a $\ZZ/4\ZZ$-extension  \cite[Thm.~$1.2.4$]{serre1992topics}. Let $\psi_K: \Gamma_{\QQ} \rightarrow \ZZ/2\ZZ$ be the associated character and suppose that $\psi_K$ is unramified outside of $\{2, \infty\}$. This lifts to a homomorphism $\varphi_K: \Gamma_{\QQ} \rightarrow \ZZ/4\ZZ$ such that $\rho_{G,v}(\varphi_{v,K}) \in \{0,2\}$. By \cite[Lem.~$2.13$]{loughransantens}, the fibre of the Iitaka fibration is $B\ZZ/2\ZZ_{\psi_K}$, and this inner twist is trivial as $\ZZ/4\ZZ$ is abelian. The weight function $w$ is now defined on $\ZZ/2\ZZ$ and takes value $1 \in \ZZ/2\ZZ$, and this is minimal, and thus $a_R(H) = b_R(H) = 1$.
\end{proof}

\noindent Here the inner twist is the inner twist as a normal subgroup, which one is not the same as an abstract inner twist. In the case where restricted ramification is imposed by $R = \{2\}$, the weight function agrees with the weight function associated to the discriminant. Let $K = \QQ(\sqrt{d})$ be a field. Then $B\ZZ/2\ZZ(K)$ corresponds to the quadratic extensions of $K$. 

\begin{lem}
The restriction of the height $H$ away from $2$ along the map $B\ZZ/2\ZZ \rightarrow B\ZZ/4\ZZ$ to a field $F \in B\ZZ/2\ZZ(K)$ is given by $H_K(F) = \Phi(K/\QQ)|N_K(\Delta_{F/K})|$ where $\Phi(K/\QQ)$ is the norm of the conductor of $K/\QQ$.

\end{lem}

\begin{proof} We have assumed that $H_v = 1$ for $v = 2,\infty$. The group $G = \ZZ/4\ZZ$ has three non-trivial irreducible representations, each of dimension $1$ as $G$ is abelian. Then away from $2$ the conductor-discriminant formula gives the result, where the factor $\Phi(K/\QQ)$ comes from the representation whose image is of order $2$ and $|N_K(\Delta_{F/K})|$ comes from the representations which have image of order $4$. \end{proof} 
\noindent Let $K = \QQ(\sqrt{d})$. As a consequence of Theorems \ref{restr. ram. gen hts} and \ref{stacky balanced hts} the leading constant satisfies  \[c(\QQ, \{2\},G,H) = \frac{1}{2}\sum_{\psi \in \im( B\ZZ/4\ZZ[\QQ] \rightarrow B\ZZ/2\ZZ[\QQ])} c(\QQ, \{0\},\ZZ/2\ZZ,H_K),\] where $H_K = H_{\QQ(\sqrt{d})}$ and the sum is over the fibres of the Iitaka fibration. We need to calculate the leading constants $c(\QQ, \{0\},\ZZ/2\ZZ,H_K)$ appearing in the sum. There is an equality of leading constants $c(\QQ, \{0\},\ZZ/2\ZZ,H_K) = c(\QQ(\sqrt{d}), \{0\},\ZZ/2\ZZ,\Phi \cdot \Delta)$. But the latter leading constant is known and is given by \[\frac{2^{-i(d)}\lim_{s\rightarrow1}(s-1)\zeta_{\QQ(\sqrt{d})}(s)}{2\Phi}\prod_{v}\left( 1 - \frac{1}{q_v}\right)\left( 1 + \frac{1}{q_v}\right)\] where the product is over places of $\QQ(\sqrt{d})$, $q_v$ is the cardinality of the residue field at $v$ and $2^{-i(d)}$ is the archimedean density where $i(d)$ is given by $0$ if $d>0$ and $1$ if $d<0$. \begin{lem}
    The only fibre of the Iitaka fibration that contributes to the sum is the fibre corresponding to the quadratic extension $\QQ(\sqrt{2})$.
\end{lem}

\begin{proof}
    The sum is over all $\ZZ/2\ZZ$-extensions of $\QQ$ that embed into a $\ZZ/4\ZZ$-extension of $\QQ$ and whose ramification type is restricted by the set $R$. A quadratic extension $\QQ(\sqrt{d})$ embeds into a $\ZZ/4\ZZ$-extension if and only if $d$ is the sum of two squares, and in particular it must be real quadratic. Furthermore, when restricted to the fibre $B\ZZ/2\ZZ$ of the Iitaka fibration the restricted ramification condition asks that for all primes $p \not\in \{2, \infty\}$ the ramification type $\rho_{G,p}(\varphi_p)$ is trivial, and so we require our $\ZZ/2\ZZ$-extension to be unramified outside of $\{2, \infty\}$. The only extension satisfying both of these conditions is $\QQ(\sqrt{2})$.
\end{proof} \noindent The conductor of $\QQ(\sqrt{2})$ is equal to $8$ and thus the leading constant is given by \[ \frac{\lim_{s \rightarrow 1}(s-1)\zeta_{\QQ(\sqrt{2})}(s)}{128} \prod_v\left(1-\frac{1}{q_v}\right)\left(1 + \frac{1}{q_v}\right),\] where the product is taken over all places of $\QQ(\sqrt{2})$. \qed

\section{Equidistribution}\label{eq section}

In this section we consider the quotient of the number of abelian number fields with restricted ramification type by the total count of abelian number fields of bounded height. This is an example of the Malle-Bhargava heuristics, which refers to the problem of counting number fields with local conditions imposed and looking at the local behaviour of the quotient of this count by the total count of number fields. In Manin's conjecture literature, the equivalent property is called equidistribution. Using our viewpoint of counting number fields as counting rational points on $BG$, we can formalise the Malle-Bhargava heuristics as the problem of equidistribution of rational points on $BG$, as in Theorem \ref{main thm}. This theorem is a formal consequence of Theorem \ref{stacky main thm intro}, as this result allows for arbitrary balanced height functions. The proof of Theorem \ref{main thm} is analogous to that of \cite[Lem.~$9.12(2)$]{loughransantens}, we sketch the details for completeness.
\begin{proof}[Proof of Theorem \ref{main thm}]
    Theorem \ref{stacky main thm intro} holds for arbitrary balanced height functions. Let $W \subset BG[\AAA_k]_{M_R(H)}$ be as in Theorem \ref{main thm}. By \cite[Prop.~$9.16$]{loughransantens}, it is enough to consider $W = \prod_{v \in S}\{\psi_v\} \times \prod_{v \not\in S}BG(\mathcal{O}_v)_{M_R(H)}$ for a finite set of places $S$ containing the bad places and some $\psi_v \in BG(k_v)$. Let $\epsilon > 0$ and define the height function $H_{\epsilon} = \prod_v H_{\epsilon, v}$ where each $H_{\epsilon,v}$ is defined as follows:  \[H_{\epsilon,v}(\varphi_v) = \begin{cases}
        H_v(\varphi_v) & \varphi \in W, \\
        \epsilon & \text{otherwise}.
    \end{cases}\] Apply Theorem \ref{stacky main thm intro} and a minor variant of Lemma \ref{lem 7.17} to this height function, and take $\epsilon \rightarrow 0 $ to obtain the correct upper bound for \[\lim_{B \rightarrow \infty} \frac{\# \{ \varphi \in BG[k]: \varphi \in W, H(\varphi) \leq B\}}{\# \{ \varphi \in BG[k] : H(\varphi) \leq B\}}.\] The lower bound is then obtained by applying the upper bound to the complement of $W$.
\end{proof} \noindent In the case where $R = G(-1)^*$, this proves \cite[Conj.~$9.15$]{loughransantens} for finite abelian groups $G$. These results imply a weaker result where we only consider finitely many local conditions, which is the version more commonly found in Malle's conjecture literature, such as in \cite{bhargavamassformulae, Wood_2010}.
\subsection{Equidistribution for general heights}

Theorem \ref{main thm} requires that $H$ be balanced, and does not hold if this condition is not met, as explained in \cite[$\S 9.6$]{loughransantens}. If one wishes to obtain a version of equidistribution for general heights, one must pass to the fibre of the Iitaka fibration, which is the homomorphism  $BG \rightarrow B(G/\langle M_R(H) \rangle)$. Since $H$ is balanced with respect to $R$ when restricted to each fibre $B\langle M_R(H) \rangle_{\psi}$, it is possible to obtain equidistribution with respect to the induced Tamagawa measure $\tau_{H_{\psi}}$. Providing a lift exists we have the following corollary to Theorem \ref{main thm}:
\begin{cor}
    Let $G$ be a finite abelian group, $\psi \in B(G/\langle M_R(H) \rangle)[k]$ and $W \subset B\langle M_R(H) \rangle_{\psi}[\AAA_k]_{M_R(H)}$ be a continuity set. Providing a lift exists, we have  \[\lim_{B \rightarrow \infty} \frac{\# \{ \varphi  \in  B\langle M_R(H) \rangle_{\psi}[k] : \varphi \in W, \: H_{\psi}(\varphi) \leq B\}}{\# \{ \varphi \in  B\langle M_R(H) \rangle_{\psi}[k] : H_{\psi}(\varphi) \leq B\}} = \frac{\tau_{H_{\psi}}(W \cap B\langle M_R(H) \rangle_{\psi}[\AAA_k]_{M_R(H)}^{\br})}{\tau_{H_{\psi}}(B\langle M_R(H) \rangle_{\psi}[\AAA_k]_{M_R(H)}^{\br})}. \]
\end{cor} 

 \bibliographystyle{amsplain}
\bibliography{ms}
\end{document}